\documentclass[a4paper,twoside,reqno]{amsart}
\newcommand{\Ueberschrift}{Anabelian geometry with \'{e}tale homotopy types}
\newcommand{\Kurztitel}{Anabelian geometry with \'{e}tale homotopy types}
\usepackage{a4,amssymb,latexsym,amsthm}
\usepackage{amsmath} \usepackage{amsopn}  \usepackage{amsfonts,dsfont}
\usepackage{mathrsfs}
\usepackage[T1]{fontenc}
\usepackage{tikz}
\usetikzlibrary{matrix,arrows}
\usepackage{tikz-cd}
\usepackage[headings]{fullpage}
\usepackage{graphicx}
\usepackage{color}
\usepackage[greek,english]{babel}
\usepackage{textgreek}
\usepackage[utf8]{inputenc}
\usepackage[pdfpagelabels, pdftex]{hyperref}
\hypersetup{
  pdftitle={},
  pdfauthor={},
  pdfsubject={},
  pdfkeywords={},
  colorlinks=true,    
  linkcolor=red,     
  citecolor=blue,     
  filecolor=blue,      
  urlcolor=blue,       
  breaklinks=true,
  bookmarksopen=true,
  bookmarksnumbered=true,
  pdfpagemode=UseOutlines,
  plainpages=false,
  unicode=true
  }

\definecolor{darklimegreen}{RGB}{31,142,8}

\usepackage{enumitem}

\newlist{enumer}{enumerate}{2}
\setlist[enumer]{label=(\roman*),align=left,labelindent=0pt,leftmargin=*,widest = (iii)}
\newlist{enumerar}{enumerate}{1}
\setlist[enumerar]{label=\arabic*.,align=left,labelindent=0pt,leftmargin=*,widest = 8.}
\newlist{enumera}{enumerate}{2}
\setlist[enumera]{label=(\arabic*),align=left,labelindent=0pt,leftmargin=*,widest = (8)}
\newlist{enumeral}{enumerate}{2}
\setlist[enumeral]{label=\rm (\alph*),align=left,labelindent=0pt,leftmargin=*,widest = (m)}

\DeclareMathOperator{\rE}{E}

\DeclareMathOperator{\rH}{H}

\DeclareMathOperator{\rc}{c}

\newcommand{\bA}{{\mathbb A}}

\newcommand{\bF}{{\mathbb F}}

\newcommand{\bP}{{\mathbb P}}
\newcommand{\bQ}{{\mathbb Q}}

\newcommand{\bZ}{{\mathbb Z}}

\newcommand{\cC}{{\mathscr C}}

\newcommand{\cM}{{\mathscr M}}

\newcommand{\cS}{{\mathscr S}}

\newcommand{\cU}{{\mathscr U}}

\newcommand{\cW}{{\mathscr W}}
\newcommand{\cX}{{\mathscr X}}
\newcommand{\cY}{{\mathscr Y}}

\newcommand{\dO}{{\mathcal O}}

\newtheorem{theorem}{Theorem}[section]
\newtheorem{corollary}[theorem]{Corollary}
\newtheorem{proposition}[theorem]{Proposition}
\newtheorem{lemma}[theorem]{Lemma}

\newtheorem{definition}[theorem]{Definition}

\theoremstyle{definition}

\newtheorem{remark}[theorem]{Remark}
\newtheorem{remarks}[theorem]{Remarks}

\newtheoremstyle{alexdef}
  {.8cm}
  {.8cm}
  {\rm }
  {}
  {\bf}
  {.}
  {.5em}
  {}
\theoremstyle{alexdef}

 \newcommand{\im}{\mathrm{im}}
 
 \DeclareMathOperator{\Spec}{Spec} 
 \newcommand{\Gal}{{G}}
 
 \newcommand{\pr}{\mathit{pr}}
 \newcommand{\Mor}{\operatorname{Hom}}
 
 \newcommand{\Aut}{\operatorname{Aut}}

 \newcommand{\colim}{\operatorname*{colim}}
 \newcommand{\Hom}{\operatorname{Hom}}
  \newcommand{\Map}{\operatorname{Map}}
 \newcommand{\Isom}{\operatorname{Isom}}
 \newcommand{\OutIsom}{\operatorname{Isom}^{\rm out}}
  \newcommand{\OutAut}{\operatorname{Aut}^{\rm out}}
   \newcommand{\OutHom}{\operatorname{Hom}^{\rm out}}
  \newcommand{\Tot}{\operatorname{Tot}}
\newcommand{\ssets}{{\sf ss}}
\newcommand{\semis}{{\rm ss}}
\newcommand{\pross}{{\sf {pro\text{-}ss}}}
\newcommand{\progp}{{\sf {pro\text{-}groups}}}
\newcommand{\Ho}{{\rm Ho}}
\newcommand{\piopen}{\pi_1\textrm{\rm -open}}
\newcommand{\open}{\textrm{\rm open}}
\newcommand{\dom}{{\rm dom}}

\newcommand{\et}{\text{\rm et}}

\newcommand{\tp}{\text{\rm top}}

\newcommand{\holim}{\operatorname{holim}}

 \newcommand{\A}{{\mathds  A}}
  \renewcommand{\P}{{\mathds  P}}
 \newcommand{\Q}{{\mathds  Q}}
 \newcommand{\Z}{{\mathds  Z}}

 \newcommand{\pro}{\mathrm{pro}\textrm{-}}

\newcommand{\lang}{\longrightarrow}
\newcommand{\ch}{\mathit{char}}

\newcommand{\F}{\mathds{F}}

\newcommand{\id}{\mathit{id}}

\newcommand{\ds}{\displaystyle}

\newcommand{\surj}{\twoheadrightarrow}
\newcommand{\inj}{\hookrightarrow}
\newcommand{\ph}{\varphi}
\newcommand{\Frob}{{\rm Frob}}
\DeclareMathOperator{\tr}{tr}
\newcommand{\R}{\mathrm R}

\DeclareMathOperator{\PGL}{PGL}

\newcommand{\hZ}{\hat{\bZ}}

\DeclareMathOperator{\Grass}{Grass}

\DeclareRobustCommand{\SkipTocEntry}[5]{}

\newcommand{\bp}[1]{\bar{#1}}
\newcommand{\defobjekt}[1]{\textbf{\boldmath #1}}
\begin{document}

\title[\Kurztitel]{\Ueberschrift}
\author{Alexander Schmidt}
\address{Alexander Schmidt, Mathematisches  Institut, Universit\"at Heidelberg, Im Neuenheimer Feld 205, 69120 Heidelberg, Germany}
\email{schmidt@mathi.uni-heidelberg.de}
\author{Jakob Stix}
\address{Jakob Stix, Institut f\"{u}r Mathematik, Goethe--Universit\"{a}t Frankfurt, Robert-Mayer-Stra{\ss}e~{6--8},
60325 Frankfurt am Main, Germany}
\email{stix@math.uni-frankfurt.de}

\date{\today}
\maketitle

\begin{quotation}
\noindent \small {\bf Abstract} --- Anabelian geometry with \'{e}tale homotopy types generalizes in a natural way classical anabelian geometry with \'{e}tale fundamental groups.
We show that, both in the classical and the generalized sense,
any point of a smooth variety over a field $k$ which is finitely generated over $\bQ$ has a fundamental system of (affine) anabelian Zariski-neighbourhoods. This was predicted by Grothendieck in his letter to Faltings \cite{grothendieck:letter}.
\end{quotation}

\section{Introduction}
\subsection{Higher anabelian geometry}
Grothendieck's anabelian philosophy \cite{grothendieck:letter} predicts the existence of a class of \emph{anabelian} varieties  $X$
that are reconstructible from their \'{e}tale fundamental group $\pi_1^\et(X, \bar{x})$. All examples of anabelian varieties known so far are of type $K(\pi,1)$, i.e., their higher \'{e}tale homotopy groups vanish.

For general varieties $X$, the homotopy theoretic viewpoint suggests to ask the modified question, whether they are reconstructible from their \emph{\'{e}tale homotopy type $X_\et$} instead of only $\pi_1^\et(X,\bar{x})$.
For varieties $X$ of type $K(\pi,1)$ this makes no difference since then $X_\et$ is weakly equivalent to the classifying space $B\pi_1^\et(X,\bar x)$.

Recall that the \emph{\'{e}tale topological type} $X_\et$ of a scheme $X$ is an object in  $\pross$, the pro-category of simplicial sets.
Any geometric point $\bp{x}$ of $X$ defines a point $\bp{x}_\et$ on $X_\et$. If $X$ is locally noetherian, the fundamental group $\pi_1(X_\et,\bp{x}_\et)$ is the usual (in the sense of \cite{sga3} X \S6) \'{e}tale fundamental group $\pi_1^\et(X,\bp{x})$ and the higher homotopy groups of $X_\et$ are the higher
\'{e}tale homotopy groups of $X$ by definition, cf.~\cite{AM}, \cite{Fr}. Isaksen \cite{Is} defined a model structure on $\pross$ and we denote the associated homotopy category by $\Ho(\pross)$. When considered as an object of $\Ho(\pross)$, we refer to $X_\et$ as the \emph{\'{e}tale homotopy type} of $X$.
For a pro-simplicial set $B$, we denote the category of morphisms to $B$ in $\Ho(\pross)$ by $\Ho(\pross) \downarrow B$.

\smallskip

In the language of \'{e}tale homotopy theory, the isomorphism form of Mochizuki's theorem on anabelian geometry of hyperbolic curves \cite{mochizuki:localpro-p} can be reformulated as follows (see Theorem~\ref{mochi} below for a more general version). Recall that a sub-$p$-adic field is a subfield of a finitely generated extension of $\Q_p$.

\begin{theorem} \label{intromoch}
Let $p$ be a prime number, $k$ a sub-$p$-adic field and $X$ and $Y$  smooth hyperbolic curves over $k$.
Then the natural map
\[
\Isom_k(X,Y) \longrightarrow \Isom_{\Ho(\pross) \downarrow k_\et}(X_\et,Y_\et)
\]
is bijective.
\end{theorem}

\subsection{Main results}
The aim of this paper is to prove Theorem~\ref{main} below, which constitutes a first step towards a generalisation of Theorem~\ref{intromoch} to higher dimensional varieties.

\begin{theorem}
\label{main}
Let $k$ be a finitely generated field extension of\/ $\Q$, and let $X$ and $Y$ be smooth, geometrically connected varieties over $k$ which can be embedded as locally closed subschemes into a product of hyperbolic curves over $k$. Then the natural map
\[
\Isom_k(X,Y) \longrightarrow \Isom_{\Ho(\pross) \downarrow k_\et}(X_\et,Y_\et)\leqno (\ast)
\]
is a split injection with a functorial retraction
\[
r : \Isom_{\Ho(\pross) \downarrow k_\et}(X_\et,Y_\et)  \longrightarrow   \Isom_k(X,Y).
\]
\end{theorem}

Theorem~\ref{main} will be proven in its refined version Theorem~\ref{thm:main-retraction}, which makes a more precise statement and, in particular, uniquely characterizes a retraction $r$. It will be this retraction $r$ that we discuss in Theorem~\ref{kernel-thm} below. Furthermore, Theorem~\ref{main-general} provides a version of Theorem~\ref{main} without the assumption of (geometrically) connectedness.

\smallskip

We stress the following weakly anabelian statement obtained as a trivial corollary.

\begin{corollary}
Let $k$ be a finitely generated field extension of\/ $\Q$ and let $X$ and $Y$ be smooth, geometrically connected varieties over $k$ which can be embedded as locally closed subschemes into a product of hyperbolic curves over $k$.

If $X_\et \cong Y_\et$ in $\Ho(\pross)\downarrow  k_\et $, then $X$ and\/ $Y$ are isomorphic as $k$-varieties.
\end{corollary}

\begin{remarks}
\begin{enumer}
\item
The reason that Theorem~\ref{main} is formulated for varieties over a finitely generated extension field of $\Q$  lies in the method of our proof, which uses techniques of Tamagawa~\cite{Ta} that we could not generalise to the more general context of sub-$p$-adic fields.

\item
By \cite{Is2}, the functor $ X\mapsto X_\et$  from smooth $k$-schemes to $\Ho(\pross)\downarrow k_\et$  factors through the $\bA^1$-homotopy category of Morel and Voevodsky \cite{MV}. In particular, it is not faithful.  However, this does not affect Theorem~\ref{main} since the schemes occurring there are $\bA^1$-local.
\end{enumer}
\end{remarks}

For \defobjekt{strongly hyperbolic Artin neighbourhoods} (see \ref{stronghyperbolicartin:def} for the definition), we can show that $(\ast)$ is a bijection.

\begin{theorem} \label{anabkpi1}
Let $X$ and $Y$ be strongly hyperbolic Artin neighbourhoods over a finitely generated field extension $k$ of\/ $\Q$.
Then the natural map
\[
\Isom_k(X,Y) \longrightarrow \Isom_{\Ho(\pross) \downarrow k_\et}(X_\et,Y_\et)
 \]
is bijective.
\end{theorem}

We denote by $\Gal_k = \pi_1(k_\et,\bar k_\et)$ the absolute Galois group of the field $k$ with respect to a fixed  separable algebraic closure $\bar k$. For a connected variety $X$ over $k$ equipped with a geometric base point $\bp{x}$ over $\bar k/k$, there is a natural augmentation map $\pi_1^\et(X,\bp{x}) \to \Gal_k$. We denote by
$\Hom_{\Gal_k}\big(\pi_1^\et(X,\bp{x}),\pi_1^\et(Y,\bp{y})\big)$ the set of
those homomorphisms that are compatible with the augmentation. Further, $\sigma \in  \pi_1^\et(Y_{\bar k},\bp{y})$ acts by composition with the inner automorphism of $\pi_1^\et(Y,\bp{y})$ given by $\sigma$, and
\[
\OutHom_{\Gal_k}\big(\pi_1^\et(X,\bp{x}),\pi_1^\et(Y,\bp{y})\big):= \Hom_{\Gal_k}\big(\pi_1^\et(X,\bp{x}),\pi_1^\et(Y,\bp{y})\big)_{\pi_1^\et(Y_{\bar k},\bp{y})}
\]
denotes the set of orbits. For geometrically connected and geometrically unibranch varieties,
this set does not depend on the chosen base points (cf.\ section~\ref{subsecpunktiert}), and we omit them from the notation.
Then there is a natural map
\[
\Mor_k(X,Y)\to \OutHom_{\Gal_k}\big(\pi_1^\et(X),\pi_1^\et(Y)\big),
\]
which factors through $\Mor_{\Ho(\pross)\downarrow k_\et}(X_\et,Y_\et)$, see Corollary~\ref{inducedhom}.
Since strongly hyperbolic Artin neighbourhoods are of type $K(\pi,1)$, Theorem~\ref{anabkpi1} can be restated  in terms of fundamental groups:

\begin{corollary} \label{anabkpi1kor}
Let $X$ and $Y$ be strongly hyperbolic Artin neighbourhoods over a finitely generated field extension $k$ of\/ $\Q$. Then the natural map
\[
\Isom_k(X,Y)  \longrightarrow \OutIsom_{\Gal_k}\big(\pi_1^\et(X),\pi_1^\et(Y)\big)
\]
is bijective.
\end{corollary}

We remark that by different techniques Hoshi proves in \cite{hoshi} \S3 a  statement similar to Corollary~\ref{anabkpi1kor} but restricted to dimension $\leq 4$.
Corollary~\ref{anabkpi1kor} implies the following statement predicted by Grothendieck in his letter to Faltings \cite{grothendieck:letter}:

\begin{corollary}\label{basisofneighbourhoods}
Let $X$ be a smooth, geometrically connected variety over a finitely generated field extension $k$ of\/ $\Q$. Then every point of $X$ has a basis of Zariski-neighbourhoods consisting of anabelian varieties, in the sense that $k$-isomorphisms between any two of these are in bijection with outer $G_k$-isomorphisms of their respective \'{e}tale fundamental groups.
\end{corollary}

The proof of Theorem~\ref{anabkpi1}, Corollary~\ref{anabkpi1kor} and Corollary~\ref{basisofneighbourhoods} will be completed in section~\ref{sec:strongArtin}.
Finally, in Theorem~\ref{main-absolute} we obtain the following absolute version of Theorem~\ref{main}.

\begin{theorem}
Let $k$ and $\ell$ be finitely generated extension fields of\/ $\Q$, and let $X/k$ and $Y/\ell$ be smooth geometrically connected varieties which can be embedded as locally closed subschemes into a product of hyperbolic curves over $k$ and $\ell$, respectively.

Then the natural map
\[
\Isom_\text{\rm Schemes}(X,Y) \longrightarrow \Isom_{\Ho(\pross)}(X_\et,Y_\et)
\]
is a split injection with a functorial retraction. If\/ $X$ and $Y$ are strongly hyperbolic Artin neighbourhoods, it is a bijection.
\end{theorem}

\subsection{On the kernel}
For general $X$, we have only partial information about the kernel of the retraction~$r$ of Theorem~\ref{main}.
In order to state our result in the general case, we need the following notation and terminology:

Let $1\to H \to G \to \Gamma \to 1$ be an exact sequence of groups. We say that $\ph\in \Aut_\Gamma(G)$ is \defobjekt{class-preserving by elements of $H$} if for every $g\in G$ there is an $h\in H$ such that $\ph(g)=hgh^{-1}$.
For a smooth, geometrically connected $k$-variety $X$, the question whether some element $\ph\in \OutAut_{\Gal_k}(\pi_1^\et(X))$ is class preserving by elements of $\pi_1^\et(X_{\bar k})$ does not depend on the chosen base point or representative in $\Aut_{\Gal_k}(\pi_1^\et(X,\bp{x}))$.

\begin{theorem}\label{kernel-thm}
Let $k$ be a finitely generated field extension of\/ $\Q$ and let $X$  be a smooth, geometrically connected variety over $k$ which can be embedded as a locally closed subscheme into a product of hyperbolic curves over $k$. Let $\gamma$ be in  the kernel of the retraction map of Theorem~\ref{main}:
\[
r : \Aut_{\Ho(\pross) \downarrow k_\et}(X_\et)  \longrightarrow   \Aut_k(X).
\]
Then the induced automorphism $\pi_1(\gamma)\in \OutAut_{\Gal_k}(\pi_1^\et(X))$ is class-preserving by elements of $\pi_1^\et(X_{\bar k})$.
\end{theorem}

Theorem~\ref{kernel-thm} will be proven in the course of section \ref{sec:controlpi1}.

\subsection{Outline}
Our first goal is to reformulate the results on anabelian geometry of hyperbolic curves proven by Mochizuki and Tamagawa in terms of the homotopy category. This change of perspective
has a big technical advantage: to formulate the result in terms of fundamental groups one has to choose  base points, and then divide out the ambiguity introduced by this choice. The formulation in the homotopy category is intrinsically base point free and more natural.

To reach this goal, we have to overcome quite a number of technical difficulties. In particular, the relation between pointed and unpointed homotopy classes of maps, which is well understood for spaces, becomes quite subtle for maps between pro-spaces.
For example, there are connected pro-spaces whose fundamental group depends on the chosen base point. We have to show that such pathologies do not occur for \'{e}tale homotopy types. Further technical problems are related to the behaviour of \'{e}tale homotopy types under base change and to the existence of classifying spaces for pro-groups. We deal with these problems in section~\ref{basicprop}, developing the necessary theory of pro-spaces in the appendix. Then we prove Theorem~\ref{intromoch} in section~\ref{sec:mochi}.

\smallskip
In section~\ref{sec:proofofmain} we prove Theorem~\ref{main}, an anabelian principle for varieties which can be embedded into a product of hyperbolic curves. Here Mochizuki's theorem in its homotopy theoretic formulation Theorem~\ref{intromoch} constitutes the first important step. We obtain a scheme morphism, however only to the ambient product space. In order to show that the morphism factors through the embedded subvariety, we use reductions over finite fields of the given varieties in a systematic way. It is here where we have to strengthen the assumption on the base field from sub-$p$-adic to finitely generated over $\Q$.

\smallskip
Unfortunately, Theorem~\ref{main} does not provide a bijection, only an injection with functorial retraction. In section~\ref{sec:controlpi1} we investigate the kernel of this retraction. Of course, we hope that it is trivial. What we can show is that elements of the kernel induce class preserving automorphisms of the fundamental group. For strongly hyperbolic Artin neighbourhoods, this suffices to show triviality. Since these are of type $K(\pi,1)$, we conclude an anabelian isomorphism result for strongly hyperbolic Artin neighbourhoods in terms of fundamental groups in the classical style of formulation.

\smallskip
The final section~\ref{sec:absolute} provides an absolute version of Theorem~\ref{main}, merging our new result with the birational anabelian geometry of the base field.

\addtocontents{toc}{\SkipTocEntry}
\subsection*{Notation and conventions}
The set of orbits for a group $G$ acting on a set $M$ is denoted by $M_G$.

\smallskip

All schemes considered are separated and locally noetherian. For an $S$-scheme $X$, a base change $X \times_S T$ is denoted by $X_T$. An \defobjekt{immersion} of schemes is the composite of an open and a closed immersion, i.e., an embedding as a locally closed subscheme.
By the phrase \defobjekt{\'{e}tale covering}  we mean finite \'{e}tale morphism, i.e., rev\^{e}tement \'{e}tale in the sense of \cite{sga1}.

\smallskip
We use the term variety (over $k$) for a scheme of finite type over the field $k$. A \defobjekt{hyperbolic curve} over a field $k$ is a geometrically connected curve $C$ over $k$  with geometrically negative \'{e}tale $\ell$-adic Euler characteristic $\chi(C_{\bar k},\bQ_\ell) < 0$ for $\ell \in k^\times$. Here $\bar k$ is an algebraic closure of $k$.

\smallskip

A \defobjekt{pro-object} in a category $\cC$ is a contravariant functor $I^\mathrm{op} \to \cC$ from some small filtering category $I$ to $\cC$. One often writes a pro-object $X$ in the form $X=(X_i)_{i\in I}$. The pro-objects in $\cC$ form a category $\pro\cC$ by setting
\[
\Mor_{\pro\cC}(X,Y)= \lim_j \colim_i \Mor_\cC(X_i,Y_j).
\]

We denote the category of simplicial sets by $\ssets$, and by $\pross$ its pro-category. Similarly, we use the notation $\ssets_*$ and $\pross_*$ for the category of pointed simplicial sets and its pro-category. We consider the closed model structure on $\pross$ defined by Isaksen \cite{Is} and its
pointed variant. We use the word \defobjekt{space} synonymous for simplicial set. For a pointed pro-space $(X,x)$, we have the homotopy groups $\pi_n(X,x)$, which are pro-groups.

For a given pro-simplicial set $B$, we denote by $\Ho(\pross)\downarrow B$ the category of morphisms to $B$ in $\Ho(\pross)$.
For a model category $\cC$ we sometimes denote morphisms in $\Ho(\cC)$ by
\[
\Mor_{\Ho(\cC)}(X,Y) = [X,Y]_{\cC}.
\]
The \'{e}tale topological type $X_\et$ of a locally noetherian scheme $X$ is the pro-space obtained by applying the functor ``connected component'' to the filtered system of rigid \'{e}tale hypercovers of~$X$, see \cite{Fr}. When considered as an object of the homotopy category $\Ho(\pross)$, we refer to $X_\et$ as the \'{e}tale homotopy type of $X$.
A geometric point $\bp{x}: \Spec(\bar K) \to X$ defines a point $\bp{x}_\et$ on $X_\et$. The \'{e}tale homotopy groups of $(X,\bp{x})$ are defined by
\[
\pi_n^\et(X,\bp{x})= \pi_n(X_\et,\bp{x}_\et).
\]
These pro-groups are pro-finite groups if $X$ is noetherian and geometrically unibranch, see \cite{AM} Thm.~11.1.
For a field $k$ and a separably closed extension field $K/k$ (given, e.g., by a geometric point of a $k$-variety), we write
\[
\Gal_k=\Gal(\bar k /k)=  \pi_1^\et(k,K)=\pi_1(k_\et,K_\et),
\]
where $\bar k$ is the separable algebraic closure of $k$ in $K$.

\section{Basic properties of  \'{e}tale homotopy types}\label{basicprop}

In this section we collect basic properties of \'{e}tale homotopy types used in this paper.
\subsection{\'{E}tale base change}
\begin{lemma} \label{lem:finiteetale}
Assume that $p: W'\to W$ is a finite \'{e}tale morphism of schemes. Then $p_\et: W'_\et \to W_\et$ is a finite covering in $\pross$ (cf.\ section~\ref{covappendix}). For any morphism $X \to W$, the natural map
\[
(X\times_WW')_\et \lang X_\et \times_{W_\et} W'_\et
\]
is an isomorphism in $\pross$.
\end{lemma}

\begin{proof}
Since the functor \emph{\'{e}tale topological type} respects connected components, we can assume that $W$ and $W'$ are connected and that $p$ is surjective of degree, say $d$. Since an \'{e}tale covering of a strictly henselian scheme splits completely, the pull-back to $W'$ of a sufficiently small \'{e}tale neighbourhood of a geometric point $\bar w$ of $W$ has $d$ connected components. Furthermore, the rigid covers of $W'$ obtained by rigid pull-back from rigid covers of $W$ are cofinal among all rigid covers of $W'$.  By recursion, the same is true for rigid hypercovers. Moreover, among the rigid hypercovers $U_\bullet$ of $W$ those with the property that for all $n$ and every connected component $V_n$ of $U_n$, the base change $V_n\times_W W'$ has $d$ connected components are cofinal. For those $U_\bullet$ the map $\pi_0(U_\bullet \times_WW')\to \pi_0(U_\bullet)$ has the lifting property of the definition of a covering in $\ssets$ (cf.\ section~\ref{covappendix}). This shows the first statement.

In order to show the second statement, let $W''\to W$ be a connected \'{e}tale Galois covering with group $G=G(W''/W)$ dominating $W'\to W$ and let $U\subset G$ be the subgroup associated with $W'$.  Then $W''$ is an \'{e}tale $G$-torsor on $W$ and, in the obvious sense, $W''_\et$ is a $G$-torsor on $W_\et$. We conclude that the natural map
\[
(X\times_WW'')_\et \lang X_\et \times_{W_\et} W''_\et
\]
is a map of $G$-torsors on $X_\et$, hence an isomorphism. The statement for $W'$  instead of $W''$ is obtained by forming the orbits of the $U$-action on both sides.
\end{proof}

We will frequently use the fact that isomorphisms in $\Ho(\pross)$ between \'{e}tale homotopy types can be base changed along finite \'{e}tale morphisms. The precise statement is the following lemma. Note that no uniqueness assertion is made on the isomorphism $\gamma'$ below.

\begin{lemma}\label{basechange}
Let $W$, $X$, $Y$ be schemes and let $f:  X \to W$, $g:Y\to W$ be morphisms. Assume that there exists
$\gamma \in \Isom_{\Ho(\pross)}(X_\et,Y_\et)$
such that  $g_\et \gamma=f_\et$ in $\Ho(\pross)$. Let $W'\to W$ be finite \'{e}tale.
Then there exists
$
\gamma'\in  \Isom_{\Ho(\pross)}\big((X\times_WW')_\et,(Y\times_WW')_\et\big)
$
such that the diagram
\[
\begin{tikzcd}[column sep=large]
X_\et  \arrow{d}[swap]{\gamma}& (X\times_WW')_\et \arrow{d}[swap]{\gamma'}\arrow{r}{}  \lar& W'_\et \dar[-, double equal sign distance]\\
Y_\et & (Y\times_WW')_\et \arrow{r}{} \lar&W'_\et
\end{tikzcd}
\]
commutes in $\Ho(\pross)$.
\end{lemma}

\begin{proof}
By Lemma~\ref{lem:finiteetale}, a finite \'{e}tale morphism $W'\to W$ induces a finite covering $W'_\et\to W_\et$ of pro-spaces, and the natural map $(X\times_WW')_\et \to X_\et \times_{W_\et} W'_\et$ is an isomorphism. The same argument applies to $Y$ and therefore the assertion follows from Proposition~\ref{basechangepross}.
\end{proof}

Varieties whose \'{e}tale homotopy types are isomorphic in $\Ho(\pross)$ have isomorphic \'{e}tale cohomology groups. A subtle point (due to the non-canonicity of the map $\gamma'$ in Lemma~\ref{basechange}) is the question whether we obtain $G_k$-module isomorphisms between the \'{e}tale cohomology groups of the base changes to $\bar k$. The next proposition gives a positive answer.

\begin{proposition}\label{unpointedisom}
Let $X$ and $Y$ be  varieties over a field $k$ and let $\bar{k}$ be a separable closure of~\/$k$.
Assume there exists an isomorphism $ X_\et \cong Y_\et$  in $\Ho(\pross) \downarrow k_\et$.
Then there exist $\Gal_{k}$-isomorphisms
\[
\rH^i_\et(X_{\bar{k}},A) \cong \rH^i_\et(Y_{\bar{k}},A)
\]
for all $i\ge 0$ and every abelian group $A$, which are moreover functorial in $A$.
\end{proposition}

\begin{proof}
Let $k'$ run through the finite subextensions of $k$ in $\bar k$. The projections $X\times_k \bar k \to X\times_k k'$ induce maps in $\pross$
\[
(X\times_k \bar k)_\et \stackrel{\alpha}{\lang} \big((X\times_k k')_\et\big)_{k'\subset \bar k} \stackrel{\beta}{\lang} \big(X_\et \times_{k_\et} k'_\et\big)_{k'\subset \bar k} \stackrel{\gamma}{\lang} X_\et \times_{k_\et} \big(k'_\et\big)_{k'\subset \bar k} \ .
\]
Because $X$ is of finite type over $k$, $\pi_0(X\times_k \bar k) \to \pi_0(X\times_k k')$ is bijective for any sufficiently large finite extension $k'$ of~$k$ in~$\bar k$. Hence the map $\alpha$ induces a bijection on $\pi_0$.
Because $X$ is quasi-compact, any \'{e}tale open cover of $X \times_k \bar k$ has a finite refinement which is defined over some finite extension $k'$ of~$k$. Hence, for any group $G$, any $G$-torsor over $X \times_k \bar k$ comes by base change from a $G$-torsor defined over some finite extension $k'$ of~$k$. Moreover,  this torsor is essentially unique modulo passage to some finite extension $k''$ of $k'$.  Therefore, for any choice of base point of $X\times_k\bar k$, the homomorphism on $\pi_1$ induced by $\alpha$ is an isomorphism. Compatibility of \'{e}tale cohomology with inverse limits (for systems of quasi-compact schemes with affine transition maps) shows that $\alpha$ also yields an isomorphism on the cohomology with values in local systems. Hence $\alpha$ is a weak equivalence by the cohomological criterion for weak equivalences \cite{Is} Prop.~18.4.

The map $\beta$ is an isomorphism by Lemma~\ref{lem:finiteetale}.
Finally, the map $\gamma$ is an isomorphism for trivial reasons. Therefore we obtain $G_k$-isomorphisms
\[
\rH^i(X_\et \times_{k_\et} \big(k'_\et\big)_{k'\subset \bar k},A) \cong \rH_\et^i(X\times_k \bar k, A)
\] for all $i$ and every abelian group $A$, which moreover are functorial in $A$. The same argument applies to  $Y$ and hence the statement of the proposition follows from Proposition~\ref{basechangepross} applied to the covering $(k'_\et)_{k'\subset \bar k}/k_\et$.
\end{proof}

\subsection{Pointed versus unpointed} \label{subsecpunktiert}

We usually consider \'{e}tale homotopy types of $k$-varieties as objects in the category of morphisms to $k_\et$ in the homotopy category of pro-spaces.
A subtle point is the relation between morphisms in the pointed and unpointed setting. We deduce from the results of Appendix~\ref{sec:pointedvsunpointed} that under suitable assumptions on the varieties this relation is essentially the same as in the classical topological situation -- at least if the base field $k$ has a  strongly center-free absolute Galois group.

\smallskip
Recall that a pro-finite group is called \defobjekt{strongly center-free} if every open subgroup has a trivial center. Important for our application is that  sub-$p$-adic fields have strongly center-free absolute Galois groups
by \cite{mochizuki:localpro-p} Lemma 15.8.

\begin{proposition} \label{prop:pointed_schemes2}
Let $X$ and $Y$ be  connected varieties over a field $k$, and assume that $Y$ is geometrically unibranch and geometrically connected.  Let $K/k$ be a separably closed extension field and let $\bp{x}: \Spec(K) \to X$ and $\bp{y}: \Spec(K)\to Y$ be geometric points (over $k$). Let $\bar k$ denote the separable closure of\/ $k$ in $K$.
\begin{enumeral}
\item
The map induced by forgetting the base points yields a surjection
\begin{equation} \label{pointedsurj}
\Mor_{\Ho(\pross_\ast) \downarrow (k_\et,\bar k_\et)}\big((X_\et,\bp{x}_\et),(Y_\et,\bp{y}_\et)\big)
\twoheadrightarrow
\Mor_{\Ho(\pross) \downarrow k_\et}(X_\et,Y_\et).
\end{equation}
In particular, if\/  $X_\et$ and $Y_\et$ are isomorphic in $\Ho(\pross) \downarrow k_\et$, then $(X_\et,\bp{x}_\et)$ and $(Y_\et,\bp{y}_\et)$ are isomorphic in $\Ho(\pross_*) \downarrow (k_\et,\bar k_\et)$.
\item
The map \eqref{pointedsurj} factors through the orbit space for the action of\/
$\pi_1^\et(Y_{\bar k},\bp{y})$ as defined in sections~\ref{sec:pointedvsunpointed1} and \ref{sec:pointedvsunpointed2} inducing a surjection
\begin{equation} \label{eq:pointedorbits}
\left(\Mor_{\Ho(\pross_\ast) \downarrow (k_\et,\bar k_\et)}\big((X_\et,\bp{x}_\et),(Y_\et,\bp{y}_\et)\big)\right)_{\pi_1^\et(Y_{\bar k},\bp{y})}
\twoheadrightarrow
\Mor_{\Ho(\pross) \downarrow k_\et}(X_\et,Y_\et).
\end{equation}
\item
If\/ $G_k$ is strongly center-free, then the map \eqref{eq:pointedorbits}  is a bijection.
\end{enumeral}
\end{proposition}

\begin{proof} By Theorem~\ref{top=prof}, $Y_\et$ and $k_\et$ are path-connected and their topological fundamental groups are the underlying abstract groups of their pro-finite \'etale fundamental groups. Since $Y$ is geometrically connected, $\pi_1^\et(Y,\bar y)\to \pi_1^\et(k,K)$ is surjective. Hence the assumptions of
Theorem~\ref{modpi1secondcase} hold with $(B,b) = (k_\et,K_\et)$.
Let $\Delta_{X,Y}$ be the group defined in Theorem~\ref{modpi1secondcase} \ref{modpi1secondcase:a}.
Since
\[
\pi_1^\et(Y_{\bar k},\bp{y}) = \ker(\pi_1^\et(Y,\bar y)\to \pi_1^\et(k,K))
\]
is a subgroup of $\Delta_{X,Y}$, the assertions (a) and (b) follow from
Theorem~\ref{modpi1secondcase}. Moreover, in order to prove (c), it remains to show that
$\pi_1^\et(Y_{\bar k},\bp{y}) = \Delta_{X,Y}$.

Let $S_X \subset G_k=\pi^\tp_1(k_\et,K_\et)$ be the stabilizer of the map $(X_\et,\bp{x}_\et) \to (k_\et,K_\et)$ in $\Ho(\pross_*)$ with respect to the $G_k$-action as defined in section~\ref{sec:pointedvsunpointed1}.
Since by Lemma~\ref{lem:monodromy_acts_by_conjugation} the induced action on $\pi^\tp_1(k_\et,K_\et)=G_k$ is by conjugation, it follows that any $g \in S_X$ centralizes the image $U$ of  $\pi_1^\et(X,\bp{x}) \to \Gal_k$. But $U$ is open and thus $g$ lies in the center of the open subgroup $\langle U, g \rangle$. Because $G_k$ is strongly center-free we conclude $g=1$, and thus $S_X= 1$.

Since by definition $\Delta_{X,Y}$ is the preimage of $S_X$,  the claim $\pi_1^\et(Y_{\bar k},\bp{y}) = \Delta_{X,Y}$ follows.
\end{proof}

We keep the assumptions of Proposition~\ref{prop:pointed_schemes2}, in particular, $\pi_1^\et(Y,\bp{y})$ is pro-finite.
Therefore, any homomorphism $\ph:  \pi_1^\et(X,\bp{x})\to \pi_1^\et(Y,\bp{y})$  factors through the pro-finite completion $\pi_1^\et(X,\bp{x})^\wedge$ of the pro-group $\pi_1^\et(X,\bp{x})$. The pro-finite completion $\pi_1^\et(X,\bp{x})^\wedge$ is the \'{e}tale fundamental group of $X$ in $\bp{x}$ in the sense of \cite{sga1}, the dependence of the base point of which is well understood. Hence
\[
\OutHom_{\Gal_k}\big(\pi_1^\et(X,\bp{x}),\pi_1^\et(Y,\bp{y})\big):= \Hom_{\Gal_k}\big(\pi_1^\et(X,\bp{x}),\pi_1^\et(Y,\bp{y})\big)_{\pi_1^\et(Y_{\bar k},\bp{y})}
\]
is independent of the chosen base points (which we will omit from the notation). There is a natural map
\[
\Mor_k(X,Y)\lang \OutHom_{\Gal_k}(\pi_1^\et (X),\pi_1^\et(Y)).
\]
If $G_k$ is strongly center-free, this map factors through the unpointed homotopy category:

\begin{corollary}\label{inducedhom} Let $k$ be a field such that $G_k$ is strongly center-free. Then, under the assumptions of Proposition~\ref{prop:pointed_schemes2}, the natural map
\[
\Mor_{\Ho(\pross_*)\downarrow (k_\et,\bar k_\et)}((X_\et,\bp{x}_\et),(Y_\et,\bp{y}_\et)) \lang \Mor_{\Gal_k}(\pi_1^\et (X,\bp{x}),\pi_1^\et(Y,\bp{y}))
\]
induces a map
\[
\Mor_{\Ho(\pross)\downarrow k_\et}(X_\et,Y_\et) \lang \OutHom_{\Gal_k}(\pi_1^\et (X),\pi_1^\et(Y)).
\]
\end{corollary}
\begin{proof}
We form  orbits for the natural $\pi_1^\et(Y_{\bar k},\bp{y})$-action on both sides and use Lemma~\ref{lem:monodromy_acts_by_conjugation}.
\end{proof}
Keeping the assumptions, we denote by
\[
\Mor_{\Ho(\pross)\downarrow k_\et}^{\piopen}(X_\et,Y_\et)
\]
the subset of those $\gamma$ such that $\pi_1(\gamma)\in \OutHom_{\Gal_k}(\pi_1^\et (X),\pi_1^\et(Y))$ has open image.
We use a similar notation in the pointed case. The bijection of Proposition~\ref{prop:pointed_schemes2} respects $\pi_1$-open maps, hence we deduce the following.

\begin{corollary} \label{cor:pointed_schemes3} Let $k$ be a field such that $G_k$ is strongly center-free. Then,
under the assumptions of Proposition~\ref{prop:pointed_schemes2},  we obtain a bijection
\[
\left(\Mor_{\Ho(\pross_\ast) \downarrow (k_\et,K_\et)}^{\piopen}((X_\et,\bp{x}_\et),(Y_\et,\bp{y}_\et))\right)_{\pi_1^\et(Y_{\bar k},\bp{y})}
\xrightarrow{\sim}
\Mor_{\Ho(\pross) \downarrow k_\et}^{\piopen}(X_\et,Y_\et).
\]
\end{corollary}

\subsection{\texorpdfstring{Varieties of type \boldmath $K(\pi,1)$}{Varieties of type K(π,1)}}
We say that a geometrically pointed, connected locally noetherian scheme $(X,\bp{x})$ is of \defobjekt{type $K(\pi,1)$} if $\pi_n^\et(X,\bp{x})$ vanishes  for all $n\geq 2$.
This is equivalent to the statement that the classifying morphism
\[
 (X_\et,\bp{x}_\et) \longrightarrow B\pi_1(X_\et,\bp{x}_\et)
\]
is an isomorphism in $\Ho(\pross_*)$, cf.\ Appendix~\ref{sec:kpi1}.
If $X$ is geometrically unibranch, then the question whether $X$ is of type $K(\pi,1)$ does not depend on the chosen base point by
Corollary~\ref{cor:pointed-unpointed-abs}.

The following lemma provides basic examples of varieties of type $K(\pi,1)$.

\begin{lemma}\label{kpi1s}
\
\begin{enumeral}
\item
\label{lemitem:curveskpi1}
Let $k$ be a field and let $C$ be a connected smooth curve over $k$. If $C$ is affine or if $C$ has genus $g(C)>0$, then $C$ is of type $K(\pi,1)$.
\item
\label{lemitem:productsofkpi1}
Assume that $k$ has characteristic zero and let $X_i$, $i=1,\ldots,n$, be geometrically connected and geometrically unibranch varieties over $k$.  If all $X_i$ are of type $K(\pi,1)$, then so is their product.
\end{enumeral}
\end{lemma}

\begin{proof}
For statement \ref{lemitem:curveskpi1} see \cite{schmidt96} Prop.~15. For the second statement, let $X$ be a connected and geometrically unibranch variety over $k$.  Then the
cohomological criterion for weak equivalences (\cite{AM}, Thm.~4.3) shows that $X$ is of type $K(\pi,1)$ if and only if the following holds for any finite abelian group $A$ and any integer $i\ge 2$:

\begin{quote}
\emph{For every \'{e}tale covering $X'\to X$ and every $\alpha\in \rH^i_\et(X',A)$ there exists an \'{e}tale covering $X''\to X'$ such that the restriction of\/ $\alpha$ to $X''$ vanishes.}
\end{quote}

\noindent
In particular, $X$ is of type $K(\pi,1)$ if and only if $X\times_k \bar k$ is, where $\bar k$ is an algebraic closure of~$k$. Hence we may assume that $k$ is algebraically closed (and of characteristic zero).
The stated fact that $X_1\times_k \cdots \times_k X_n$ is of type $K(\pi,1)$, if all $X_i$ are, now easily follows from the fact that
($k$ algebraically closed and characteristic zero)
\[
\pi_1^\et(X_1\times_k \cdots \times_k X_n)\cong \pi_1^\et(X_1)\times \cdots \times  \pi_1^\et(X_n)
\]
(see \cite{sga1}, Exp.\ XIII, Prop.~4.6) and from the K\"{u}nneth-formula for \'{e}tale cohomology (\cite{sga4h} (Th.\ finitude),
Cor.~1.11).
\end{proof}

In characteristic zero, the $K(\pi,1)$-property is preserved in elementary fibrations. Recall that an elementary fibration $X\to Y$ is the complement in a smooth proper curve $\bar X \to Y$ with geometrically connected fibres of a divisor $D\subset \bar X$ which is finite and \'{e}tale over $Y$, and such that the fibres of $X \to Y$ are affine curves.

\begin{proposition} \label{prop:elemfibrationpreservesKpi1}
Let $f: X \to Y$ be an elementary fibration of smooth varieties over a field $k$ of characteristic zero. If\/ $Y$
is of type $K(\pi,1)$, then so is $X$.
\end{proposition}

\begin{proof}
Choose geometric points $\bp{x}$ and $\bp{y}$ of $X$ and $Y$ with $\bp{y}=f(\bp{x})$. By  assumption, the schemes $X$ and $Y$ and $X_{\bar y}$ are smooth, in particular, geometrically unibranch. Therefore, by \cite{AM}, Thm.~11.1, their \'{e}tale homotopy types are pro-finite.
Let $\cY = (Y(i))_{i \in I}$ be the inverse system  of the pointed (finite) \'{e}tale Galois coverings $Y(i)\to (Y,\bp{y})$ and let
$\cX = \big((X,\bp{x})\times_{(Y,\bp{y})} Y(i)\big)_{i\in I}$
 be its pull-back to $(X,\bp{x})$.  Then, by \cite{Fr}, Thm.~11.5 (with $L$ the set of all prime numbers), we have a long exact sequence
\[
\cdots \to \pi_n^\et(X_{\bar y}, \bp{x}) \to \pi_n^\et(\cX) \to \pi_n^\et(\cY) \to \pi_{n-1}^\et(X_{\bar y}, \bp{x}) \to \cdots.
\]
By Lemma~\ref{lem:finiteetale}, the natural maps
\[
\pi_n^\et(\cX) \to \pi_n^\et(X,\bp{x}), \quad \pi_n^\et(\cY) \to \pi_n^\et(Y,\bp{y})
\]
are isomorphisms for $n\geq 2$.
For $n=1$ we have $\pi_1^\et(\cY)=1$ and the short exact sequence
\[
1 \to \pi_1^\et(\cX) \to \pi_1^\et(X,\bp{x}) \to \pi_1^\et(Y,\bp{y})\to 1.
\]
We therefore obtain the long exact homotopy sequence
\[
\cdots \to \pi_n^\et(X_{\bar y}, \bp{x}) \to \pi_n^\et(X,\bp{x}) \to \pi_n^\et(Y,\bp{y})  \to \pi_{n-1}^\et(X_{\bar y}, \bp{x})  \to \cdots,
\]
showing the statement of the proposition in view of Lemma~\ref{kpi1s} \ref{lemitem:curveskpi1}.
\end{proof}

The following lemma shows that morphisms in the homotopy category to products of varieties of type $K(\pi,1)$ can be given component-wise. We will use this fact in an essential way in the case of morphisms to products of hyperbolic curves.
\begin{lemma}\label{prod_in_homcat}
Let $k$ be a field of characteristic zero such that $\Gal_k$ is strongly center-free. Let $X$ be a  variety over $k$ and let $Y_1$, $Y_2$ be geometrically connected, geometrically unibranch varieties of type $K(\pi,1)$ over~$k$. Then the natural map
\[
\Hom_{\Ho(\pross)\downarrow k_\et}(X_\et, (Y_1\times_k Y_2)_\et) \to \Hom_{\Ho(\pross)\downarrow k_\et}(X_\et, Y_{1,\et}) \times \Hom_{\Ho(\pross)\downarrow k_\et}(X_\et, Y_{2,\et})
\]
is bijective.
\end{lemma}

\begin{proof}
We may assume that $X$ is connected.
Let $\bar k/k$ be an algebraic closure and choose geometric points of $\bp{x} \in X$, $\bp{y}_i \in Y_i$ over $\bar k /k$ and set $\bp{y} = (\bp{y}_1, \bp{y}_2) \in Y = Y_1 \times_k Y_2$.  Since $k$ has characteristic zero,
we have $\pi_1^\et(Y_{\bar k},\bp{y})\cong \pi_1^\et(Y_{1,\bar k},\bp{y}_1) \times \pi_1^\et(Y_{2,\bar k},\bp{y}_2)$, see \cite{sga1} Exp.\ XIII Prop.~4.6, hence
\[
\pi_1^\et(Y,\bp{y})\cong \pi_1^\et(Y_1,\bp{y}_1) \times_{\Gal_k}\pi_1^\et(Y_2,\bp{y}_2).
\]
We conclude that for every pro-group $\pi$ with augmentation $\pi \to \Gal_k$ the natural map
\[
\Hom_{\Gal_k}\big(\pi, \pi_1^\et(Y,\bp{y})\big) \lang \Hom_{\Gal_k}\big(\pi,\pi_1^\et(Y_1,\bp{y}_1)\big) \times \Hom_{\Gal_k}\big(\pi,\pi_1^\et(Y_2,\bp{y}_2)\big)
\]
is bijective. Hence, because $Y$ is also of type $K(\pi,1)$ by Lemma~\ref{kpi1s}~\ref{lemitem:productsofkpi1},
Proposition~\ref{prop:BGisKpi1} implies that
\begin{multline*}
\Hom_{\Ho(\pross_*) \downarrow(k_\et,\bar k_\et)}\big((X_\et,\bp{x}_\et),(Y_\et,\bp{y}_\et)\big) \lang\\
\Hom_{\Ho(\pross_*) \downarrow(k_\et,\bar k_\et)} \big((X_\et,\bp{x}_\et),(Y_{1,\et}, \bp{y}_{1,\et})\big) \times
\Hom_{\Ho(\pross_*) \downarrow(k_\et,\bar k_\et)} \big((X_\et,\bp{x}_\et),(Y_{2,\et}, \bp{y}_{2,\et})\big)
\end{multline*}
is bijective. Considering sets of orbits for the $\pi_1^\et(Y_{\bar k},\bp{y})\cong \pi_1^\et(Y_{1,\bar k},\bp{y}_1)\times \pi_1^\et(Y_{2,\bar k},\bp{y}_2)$-action on both sides, we obtain the result by Proposition~\ref{prop:pointed_schemes2}.
\end{proof}

\section{Homotopy theoretic formulation of Mochizuki's theorem}\label{sec:mochi}

For smooth, connected $k$-varieties $X$ and $Y$ let
\[
\Mor_k^\dom(X,Y)
\]
denote the set of dominant $k$-morphisms from $X$ to $Y$. Every dominant morphism $X\to Y$ defines a morphism $X_\et\to Y_\et$
that is $\pi_1$-open.  Similarly, for $G_k$-augmented pro-finite groups $\Gamma$ and $\Delta$ we let
\[
\Mor^{\open}_{\Gal_k}( \Gamma, \Delta)
\]
denote the set of continuous $G_k$-homomorphisms $\Gamma \to \Delta$ with open image. Mochizuki proved the following:

\begin{theorem}[\cite{mochizuki:localpro-p} Thm.~A] \label{mochi0}
Let $p$ be a prime number, $k$ a sub-$p$-adic field, $X$ a smooth,  connected $k$-variety and\/ $Y$  a smooth hyperbolic curve over $k$.
Then, for any choice of geometric base points, the natural map
\[
\Mor_k^\dom(X,Y) \longrightarrow \Mor^\open_{\Gal_k}(\pi_1^\et(X,\bp{x}),\pi_1^\et(Y,\bp{y}))_{\pi_1^\et(Y_{\bar k},\bp{y})}
\]
is bijective.
\end{theorem}

We reformulate Theorem~\ref{mochi0} in the language of homotopy theory as follows.

\begin{theorem} \label{mochi}
Let $p$ be a prime number, $k$ a sub-$p$-adic field, $X$ a smooth, connected $k$-variety and\/ $Y$  a smooth hyperbolic curve over $k$.
Then the natural map
\[
\Mor_k^\dom(X,Y) \longrightarrow \Mor_{\Ho(\pross) \downarrow k_\et}^{\piopen}(X_\et,Y_\et)
\]
is bijective.
\end{theorem}

\begin{proof}
We choose an algebraic closure $\bar{k}$ of $k$,
and further choose base points $\bp{x} \in X(\bar{k})$ and $\bp{y} \in Y(\bar{k})$ compatible with the base point $\bar k_\et$ of $k_\et$.
In the commutative diagram
\[
\begin{tikzcd}[column sep=small]
\Hom_k^\dom(X,Y)  \arrow{d} \arrow{rr} && \Hom^\open_{\Gal_k}(\pi_1^\et(X,\bp{x}),\pi_1^\et(Y,\bp{y}))_{\pi_1^\et(Y_{\bar k},\bp{y})}
 \\
\Hom_{\Ho(\pross) \downarrow  k_\et}^{\piopen}(X_\et,Y_\et)
 \arrow{rr}   &&  \Hom_{\Ho(\pross_\ast) \downarrow (k_\et, \bar k_\et)}^{\piopen}\big((X_\et,\bp{x}_\et),(Y_\et,\bp{y}_\et)\big)_{\pi_1^\et(Y_{\bar k},\bp{y})}
   \arrow[swap]{u}{\pi_1(-)}
\end{tikzcd}
\]
the bottom arrow is the inverse to the bijection of Proposition~\ref{prop:pointed_schemes2}. The arrow marked $\pi_1(-)$ is a bijection
because  hyperbolic curves are of type $K(\pi,1)$ by Lemma~\ref{kpi1s}~\ref{lemitem:curveskpi1}, hence Proposition~\ref{prop:BGisKpi1} applies. The restrictions to ($\pi_1$-)open maps are compatible.

Now the claim of the theorem is equivalent to the bijectivity of the top arrow, which is the statement of Theorem~\ref{mochi0}.
\end{proof}

\section{The retraction}\label{sec:proofofmain}

In this section we prove Theorem~\ref{main}.
\subsection{Counting points in closed fibres}
\label{sec:countingpoints}
We consider a normal noetherian scheme $S$ with geometric generic point  $\bar \eta$ over the generic point $\eta$. Let $s \in S$ be a closed point with finite residue field $\bF_s = \kappa(s)$ of cardinality $N(s) = |\bF_s|$, and let $\bar s$ be a geometric point over $s$. A choice of an \'{e}tale path between $\bar s$ and $\bar \eta$ leads to a homomorphism
\[
\Gal_{\kappa(s)} = \pi_1^\et(s,\bar s) \to \pi_1^\et(S,\bar s) \cong \pi_1^\et(S,\bar \eta),
\]
by means of which the arithmetic Frobenius $\ph_s \in \Gal_{\kappa(s)}$ acts on $\pi_1^\et(S,\bar \eta)$-modules.

\begin{proposition} \label{prop:countpointsclosedfibre}
In the above situation, let $\ell$ be a prime number invertible on $S$.
Let $\bar f: \bar X \to S$  be a proper, smooth, equidimensional morphism of relative dimension $d$, and let  $X \subseteq \bar X$ be the open complement of a strict normal crossing divisor $D  = \bigcup_{\alpha=1}^n D_\alpha \inj \bar X$ relative to $S$ with $D_\alpha/S$ smooth relative divisors for all\/ $\alpha =1, \ldots, n$.

Then the $\Gal_{\kappa(\eta)}$-action on $\rH^i(X_{\bar \eta},\bQ_\ell)$ factors through $\pi_1^\et(S,\bar \eta)$ for all $i \geq 0$, and the resulting action of\/ $\ph_s$ computes the number of\/ $\bF_s$-rational points of the fibre $X_s$ by
\[
|X_s(\bF_s)| = N(s)^d \cdot \sum_{i = 0}^{2d}  (-1)^i \tr \big(\ph_s | \rH^i_\et(X_{\bar \eta},\bQ_\ell)\big).
\]
\end{proposition}
\begin{proof}
We denote by $j: X \inj \bar X$ the open immersion. For a finite subset $J \subseteq \{1,\ldots,n\}$, we set
\[
\begin{tikzcd}
\bar f_J : D_J := \bigcap_{\alpha \in J} D_\alpha  \arrow[hook]{r}{i_J} & \bar X \arrow{r}{\bar f} & S
\end{tikzcd}
\]
which is proper and smooth. By  \cite{workofgabber}, Expos\'{e} XVI, Cor.\ 3.1.3, there are isomorphisms for all $b \geq 0$
\[
\R^b j_\ast \bQ_\ell = \bigwedge^b \big(\R^1 j_\ast \bQ_\ell \big) \cong \bigoplus_{|J| = b} i_{J,\ast} \bQ_\ell(-b).
\]
Therefore all sheaves occurring in the $\rE_2$-page of the Leray spectral sequence for $f = \bar f  j : X \to S$
\[
\rE_2^{ab} =  \R^a \bar f_\ast \big( \R^b j_\ast \bQ_\ell \big) \cong \bigoplus_{|J| = b} \R^a \bar f_{J,\ast} \bQ_\ell(-b)
\Longrightarrow \R^{a+b} f_\ast(\bQ_\ell)
\]
are smooth \'{e}tale sheaves on $S$ by \cite{sga4h} (Arcata V) Thm.~3.1.
Hence also the limit terms $\rE_\infty^{ab}$ and, furthermore, all $\R^i f_\ast \bQ_\ell$ are smooth $\bQ_\ell$-sheaves on $S$.
Relative Poincar\'{e} duality shows that also the sheaves $\R^i f_{!} \bQ_\ell$ are smooth sheaves on $S$. Hence proper base change and cospecialisation yield a $\Gal_{\kappa(s)} \to \pi_1^\et(S,\bar \eta)$-equivariant isomorphism
\[
\rH^i_{\rc}(X_{\bar s},\bQ_\ell)  = \big(\R^i f_{!} \bQ_\ell\big)_{\bar s} \xrightarrow{\sim} \big(\R^i f_{!} \bQ_\ell\big)_{\bar \eta} = \rH^i_{\rc}(X_{\bar \eta},\bQ_\ell).
\]
Poincar\'{e}-duality yields a $\Gal_{\kappa(s)}$-equivariant perfect pairing
\[
\rH^i_\et(X_{\bar s},\bQ_\ell) \times \rH^{2d-i}_{\rc}(X_{\bar s},\bQ_\ell) \lang \bQ_\ell(-d),
\]
and similarly for $X_{\bar \eta}$, which leads to $\Gal_{\kappa(s)}$-module isomorphisms
\[
 \rH^i_\et(X_{\bar s},\bQ_\ell) \cong \Hom\big(\rH^{2d - i}_{\rc}(X_{\bar s},\bQ_\ell),\bQ_\ell(-d)\big) \cong \Hom\big(\rH^{2d - i}_{\rc}(X_{\bar \eta},\bQ_\ell),\bQ_\ell(-d)\big) \cong  \rH^i_\et(X_{\bar \eta},\bQ_\ell).
\]
The arithmetic Frobenius $\ph_s$ acts by  transport of structure on \'{e}tale cohomology (with compact support)
as the inverse of the action by the geometric Frobenius $\Frob_s$.
The Lefschetz trace formula for the number of rational points on $X_s$ therefore implies
\begin{align*}
|X_s(\bF_s)| & = \sum_{i = 0}^{2d} (-1)^i \tr\big(\Frob_s | \rH^i_{\rc}(X_{\bar s},\bQ_\ell) \big)  \\
& =  N(s)^d \cdot \sum_{i = 0}^{2d} (-1)^i \tr\big(\ph_s | \rH^i_\et(X_{\bar s},\bQ_\ell)\big)
 = N(s)^d \cdot \sum_{i = 0}^{2d}  (-1)^i \tr\big(\ph_s | \rH^i_\et(X_{\bar \eta},\bQ_\ell)\big). \qedhere
\end{align*}
\end{proof}

\subsection{Factor-dominant embeddings}

Let $Y$ be a locally closed subscheme in a product of smooth, geometrically connected curves $C_i$ over~$k$
\[
\iota : Y \inj W=C_1\times \cdots \times C_n.
\]
We denote the projections by $p_i: W \to C_i$.
\begin{definition}
We say that $\iota$ is \defobjekt{factor-dominant} if $p_i \iota$ is dominant for all~$i=1,\ldots, n$.
\end{definition}

Assume that $Y$ is geometrically connected and geometrically reduced over $k$. Then the composition $p_i  \iota : Y \to C_i$   is either dominant or constant. If $p_i \iota$ is constant,  the image of $Y$ in $C_i$ is a $k$-rational point. Hence we can remove all factors $C_i$ with $p_i \iota$ constant from $W$ to obtain a factor-dominant immersion.

\begin{proposition} \label{prop:geometryfromMochizuki}
Let $p$ be a prime number and let $k$ be a sub-$p$-adic field.

Let\/  $\iota:  Y \hookrightarrow W=C_1\times \cdots \times C_n$ be a factor-dominant immersion of a geometrically connected and geometrically unibranch variety\/ $Y$ over $k$ into a product of hyperbolic curves $C_i$ and let $X$ be a smooth connected variety over $k$.

Then, for any $\pi_1$-open morphism $\gamma: X_\et \to Y_\et $
in $\Ho(\pross) \downarrow k_\et$  there is a unique morphism of\/ $k$-varieties $f : X \to W$
such that the following diagram commutes in $\Ho(\pross) \downarrow k_\et$:
\[
\begin{tikzcd}[column sep=small]
& X_\et \arrow{dl}[swap]{\gamma}\arrow[dashrightarrow]{dr}{f_\et} & \\
Y_\et \arrow{rr}[swap]{\iota_\et} & &  W_\et .
\end{tikzcd}
\]
\end{proposition}

\begin{proof}
In the degenerate case $n=0$, we have $Y = \Spec(k) = W$ and the structure map $f: X \to \Spec(k)$ is the required morphism.  Otherwise, by Theorem~\ref{mochi}, there are unique $k$-morphisms $f_i: X \to C_i$, for $i = 1, \ldots, n$,  with
\[
(f_i)_\et =  (p_i \iota)_\et  \gamma
\]
in $\Ho(\pross)\downarrow k_\et$. These together define a $k$-morphism $f = (f_i) : X  \to W$. An inductive application of Lemma~\ref{prod_in_homcat} shows that  $f_\et = \iota_\et  \gamma$ in $\Ho(\pross)\downarrow k_\et$. The uniqueness of such an $f$ is obvious.
\end{proof}

\subsection{The key argument}

Next we show that the morphism constructed in Proposition~\ref{prop:geometryfromMochizuki} factors through the subvariety $Y \inj W$  if $\gamma$ is an isomorphism in $\Ho(\pross)$.

\begin{proposition} \label{aufX}
Let $k$ be a finitely generated extension field of\/ $\Q$.  Let
$\iota:  Y \hookrightarrow W=C_1\times \cdots \times C_n$ be a smooth, locally closed subscheme in a product of hyperbolic curves over $k$, $X$ a smooth variety over $k$ and $f:  X \to W$ a $k$-morphism.

Assume there exists $\gamma \in \Isom_{\Ho(\pross)}(X_\et,Y_\et)$ such that the diagram
\begin{equation*} \label{eq:firstdiagram}
\begin{tikzcd}[column sep=small]
& X_\et \arrow{dl}[swap]{\gamma}\arrow{dr}{f_\et} & \\
Y_\et \arrow{rr}[swap]{\iota_\et} & & W_\et
\end{tikzcd}
\end{equation*}
commutes in $\Ho(\pross)$. Then $f$ factors through $\iota$,  i.e., there exists a unique
morphism $g:X \to Y$ such that the diagram
\[
\begin{tikzcd}[column sep=small]
& X \arrow[dashrightarrow]{dl}[swap]{g}\arrow{dr}{f} & \\
Y \arrow[hook]{rr}[swap]{\iota} & & W
\end{tikzcd}
\]
commutes.
\end{proposition}

\begin{remark}
The first diagram in Proposition~\ref{aufX} remains commutative after replacing $\gamma$ by $g_\et$, however, we do not claim  that $\gamma=g_\et$ in $\Ho(\pross)$.
\end{remark}

\begin{proof}[Proof of Proposition~\ref{aufX}] The question whether $f$ factors through $Y$ can be checked after base change to an \'{e}tale covering of $W$. Note that the assumption on the existence of $\gamma$ is preserved by such a base change due to Lemma~\ref{basechange}.

Since the $C_i$ are hyperbolic, there are  hyperbolic curves $C'_i$ over a common finite separable extension $k'/k$ with smooth compactification of genus $\geq 2$ and \'{e}tale coverings $C'_i \to C_i \times_k k' \to C_i$. With $W' = C'_1 \times_{k'} \ldots \times_{k'} C'_n$ we base change by the natural product covering $W' \to W$ and replace $k$ by $k'$.
We therefore may assume that all $C_i$ have compactifications of genus $\ge 2$ and then replace the $C_i$ by their smooth compactifications.

Since $k$ has characteristic $0$, we find smooth compactifications $\bar X$ and $\bar Y$ of $X$ and of $Y$ such that the boundaries are simple normal crossing divisors and $f$ and $\iota$ extend to morphisms from $\bar X$ and $\bar Y$ to $W$. Now choose a regular connected scheme $S$ of finite type over $\bZ$ with function field $k$ such that the whole situation extends over $S$. Then everything follows from Proposition~\ref{prop:tamagawa-argument} below.
\end{proof}

\begin{proposition}\label{prop:tamagawa-argument}
Let $S$ be a regular connected scheme of finite type over $\bZ$ with generic point $\eta \in S$, and let $C_i\to S$,  for $i=1,\ldots,n$, be proper smooth relative curves with geometrically connected fibres of genus $\ge 1$.

Let $\iota : Y \hookrightarrow W =C_1 \times_S \ldots \times_S C_n$ be a locally closed subscheme which is smooth as an $S$-scheme, and let  $f: X \to W$ be an $S$-morphism with $X/S$ smooth. Furthermore, assume that $X \to  S$ and $Y \to S$ have nice relative compactifications as used in Proposition~\ref{prop:countpointsclosedfibre}, which even map to $W$.

Assume there exists $\gamma\in \Isom_{\Ho(\pross)}((X_\eta)_\et,(Y_\eta)_\et)$ such that $(f_\eta)_\et = (\iota_\eta)_\et  \gamma$ in $\Ho(\pross)$. Then $f$ factors through\/ $Y$:
\[
\begin{tikzcd}[column sep=small]
& X \arrow[dashrightarrow]{dl}[swap]{g}\arrow{dr}{f} & \\
Y \arrow[hook,swap]{rr}{\iota} & & W.
\end{tikzcd}
\]
\end{proposition}
\begin{proof}
The assertion is trivial in the degenerate case $n=0$, since then $Y = W = S$. We therefore may assume $n \geq 1$.
The morphism $f$ factors through $Y$ if and only if the immersion $\iota_X: Y \times_W X \inj X$ is an isomorphism.
Because $X$ is reduced and $\iota_X$ is an immersion, it suffices to show that $\iota_X$ is surjective (on points). Since $X$ is of finite type over $\bZ$ and therefore Jacobson by \cite{EGA4} Cor.~10.4.6, the map $\iota_X$ is surjective if all closed points are in the image.

By \cite{EGA4} Lem.~10.4.11.1, every closed point of $X$ has a finite residue field. We therefore have to show that for every finite field $\bF$ (more precisely $\Spec(\bF) \to S$) and every point $x \in X(\bF)$ we have
\[
f(x) \in \im\big(\iota: Y(\bF) \to W(\bF)\big).
\]
Let $s \in S$ be the closed point of the base under $x$.  The residue field $\kappa(s)$ is a subfield of $\bF$. Hence there exists an \'{e}tale morphism $S'\to S$ with $S'$ connected and a point $s'\in S'$ mapping to $s$ and with $\kappa(s')=\bF$.
The question whether $f(x)$ lies in the image of $\iota: Y(\bF) \to W(\bF)$ can be decided after base change along $S'\to S$.

We denote by $f' : X' \to W'$ and $\iota' : Y' \inj W'$ the base change of $f$ and $\iota$ along $S' \to S$.
Since the function field $k'$ of $S'$ is a finite separable extension of $k=\kappa(\eta)$,
the existence of  $\gamma: (X_\eta)_\et \stackrel{\sim}{\to} (Y_\eta)_\et$ in $\Ho(\pross)$ with $(f_\eta)_\et = (\iota_\eta)_\et \gamma$ in $\Ho(\pross)$ implies the existence of  an isomorphism
\[
\gamma': (X_{k'})_\et \to  (Y_{k'})_\et
\]
in $\Ho(\pross)$ with $(f'_\eta)_\et = (\iota'_\eta)_\et  \gamma'$ in $\Ho(\pross)$ by Lemma~\ref{basechange}. Hence the base change along $S'\to S$ preserves all assumptions and we may assume that $\kappa(s)=\bF$ without loss of generality.

Let $f(x)=(w_1,\ldots,w_n)$, $w_i \in C_{i,s}(\F)$. Fix a decomposition group $H_i$ of $w_i$ in $\pi_1^\et(C_{i,s})$ and let $w_i'$ be another rational point of $C_{i,s}$.
By \cite{Ta} Cor.~2.10, the decomposition subgroups in $\pi_1^\et(C_{i,s})$ of different rational points of $C_{i,s}$ are not conjugate. Hence we find an open subgroup $U_i(w_i') \subset \pi_1^\et(C_{i,s})$ such that $U_i(w_i')$ contains $H_i$ but none of the (conjugate) decomposition groups of $w_i'$. Let $U_i$ be the intersection of the groups $U_i(w_i')$, where $w_i'$ runs through the finitely many rational points of $C_{i,s}$ different from $w_i$. Then the connected \'{e}tale covering of $C_{i,s}$ corresponding to $U_i$ has a rational point over $w_i$ but no rational point lying over any other rational point of $C_{i,s}$.
Taking the product of these coverings,
we find an \'{e}tale covering $h:W_s' \to W_s$  such that
\[
h(W'_s(\F))=\{ f(x)\}.
\]
Hence, if the base change $Y'_s = Y_s \times_{W_s} W'_s$ has an $\F$-rational point, then $f(x) \in Y_s(\bF)$.

\smallskip

By \cite{Art} Thm.~3.1, we can find a Nisnevich neighbourhood $(T,t) \to (S,s)$ and an \'{e}tale covering $h_T: W' \to W_T$ extending $h: W'_s \to W_s$.
Then, by the same arguments as above, the  base change along $T\to S$ preserves our assumptions
and we may assume that $S=T$.

Let us denote by $X' \to X$ (resp.\ $Y' \to Y$) the induced \'{e}tale coverings by means of $f$ (resp.\ $\iota$) from $W' \to W$. We still have an isomorphism $\gamma' : (X'_\eta)_\et \to (Y'_\eta)_\et$ in $\Ho(\pross)$ with
$(f'_\eta)_\et = (\iota'_\eta)_\et  \gamma'$ in $\Ho(\pross)$
by Lemma~\ref{basechange} and $X' \to S$ (resp.\ $Y' \to S$) keep having a nice relative compactification, because the respective \'{e}tale coverings are induced by an \'{e}tale covering of the proper scheme $W/S$. Note that because of $(f'_\eta)_\et = (\iota'_\eta)_\et  \gamma'$ the morphism $\gamma'$ lies in fact in $\Ho(\pross) \downarrow k'_\et$. 

\smallskip

Let $\bF$ have $q$ elements and let $d_X$ be the relative dimension of $X/S$, and $d_Y$ the relative dimension of $Y/S$.
Propositions~\ref{prop:countpointsclosedfibre} and \ref{unpointedisom} show that
\begin{align*}
|Y'_s(\bF)| & = q^{d_Y} \cdot \sum_{i = 0}^{\infty}  (-1)^i \tr\big(\ph_s | \rH^i(Y'_{\bar \eta},\bQ_\ell) \big)  = q^{d_Y} \cdot \sum_{i = 0}^{\infty}  (-1)^i \tr\big(\ph_s | \rH^i(X'_{\bar \eta},\bQ_\ell) \big) = q^{d_Y-d_X} \cdot |X'_s(\bF)|.
\end{align*}
Since  $X'_s(\bF)= X_s(\bF) \times_{W_s(\bF)} W'_s(\bF)$ is non-empty, we obtain  $Y'_s(\bF)\neq \varnothing$.
\end{proof}

\subsection{Independence, functoriality and retraction}

We now complete the proof of Theorem~\ref{main}. We state a more precise version of it.

\begin{theorem} \label{thm:main-retraction}
Let $k$ be a finitely generated field extension of\/ $\Q$ and let $X$ and $Y$ be smooth geometrically connected varieties over $k$ which can be embedded as locally closed subschemes into a product of hyperbolic curves over~$k$. Then the natural map
\[
(-)_\et : \Isom_k(X,Y) \longrightarrow \Isom_{\Ho(\pross) \downarrow k_\et}(X_\et,Y_\et)
\]
admits a unique functorial retraction
\[
r : \Isom_{\Ho(\pross) \downarrow k_\et}(X_\et,Y_\et) \longrightarrow  \Isom_k(X,Y),
\]
with the following properties:
\begin{enumeral}
\item
\label{thmitem:retraction}
\emph{Retraction:} for all $k$-isomorphisms $g: X \stackrel{\sim}{\to} Y$ we have
\begin{equation*} \label{eq:retractproperty}
r(g_\et) = g.
\end{equation*}

\item
\label{thmitem:r_functorial}
\emph{Functoriality:} Let $Z$ be a further geometrically connected variety over $k$ which can be embedded as a locally closed subscheme into a product of hyperbolic curves over~$k$.
Then for isomorphisms $\gamma_1: X_\et \stackrel{\sim}{\to} Y_\et$ and $\gamma_2 : Y_\et \stackrel{\sim}{\to} Z_\et$ in $\Ho(\pross) \downarrow k_\et$ we have
\begin{equation*} \label{eq:functorialityofretraction}
r(\gamma_2 \gamma_1) = r(\gamma_2)r(\gamma_1).
\end{equation*}

\item
\label{thmitem:r_and_map_to_hyperbolic_curve}
\emph{Maps to hyperbolic curves:}
If\/ $\gamma: X_\et\stackrel{\sim}{\to} Y_\et$ is an isomorphism in $\Ho(\pross) \downarrow k_\et$ and $h: Y \to C$ is a dominant $k$-morphism to a hyperbolic curve $C$, then
\[
h_\et  r(\gamma)_\et = h_\et  \gamma
\]
in $\Ho(\pross) \downarrow k_\et$.
\end{enumeral}
\end{theorem}
\begin{proof}
Let $\iota : Y \inj W = C_1 \times \ldots \times C_n$
be an embedding into a product of hyperbolic curves. After removing factors, we can assume that $\iota$ is factor-dominant. Starting from an isomorphism
\[
\gamma : X_\et \xrightarrow{\sim} Y_\et
\]
in $\Ho(\pross) \downarrow k_\et$, Proposition~\ref{prop:geometryfromMochizuki} shows the existence of a unique $k$-morphism $f: X \to W$ such that $f_\et=\iota_{\et} \gamma$. By Proposition~\ref{aufX}, the map $f$ factors as
$
f= \iota r(\gamma)
$
for a unique $k$-morphism
\[
r(\gamma): X \to Y.
\]
This constructs the $k$-morphism $r(\gamma)$ of which we will  later show that it is an isomorphism. Immediately from the construction we deduce
\begin{equation} \label{eq:basic_compatibility_with_pi1}
(\iota r(\gamma))_\et = f_\et = \iota_{\et} \gamma .
\end{equation}
Denoting the projections by $p_i: W\to C_i$, we obtain
\[
(p_i\iota r(\gamma))_\et = (p_i\iota)_\et \gamma,
\]
hence $(p_i\iota r(\gamma))_\et$ is $\pi_1$-open and $p_i\iota r(\gamma): X \to C_i$ is dominant for $i=1,\ldots,n$.
\smallskip

We first prove property \ref{thmitem:retraction}. Let us assume $\gamma = g_\et : X_\et \to Y_\et$ arises from a $k$-isomorphism $g : X \stackrel{\sim}{\to} Y$. By construction, the auxiliary map is $f = \iota g : X \to W$, which factors through $Y$ as $g: X \to Y$. Uniqueness of the auxiliary map and of the factorization shows
\[
r(g_\et) = g.
\]

We secondly show that $r(\gamma)$ is independent of the immersion $\iota$.
Let  $\iota' : Y \inj W' = \prod_j C'_j$ be another factor-dominant  embedding  into a product of hyperbolic curves. We obtain from the construction above  a unique map $f' : X \to W'$ and a factorization $g' : X \to Y$. Applying the construction a third time, namely to the product  $(\iota,\iota') : Y \inj W \times W'$, yields the map $(f,f') : X \to W \times W'$ and a further factorization $h: X \to Y$.
Projecting to both factors in $(f,f')$ we deduce $g = h = g'$, and $r(\gamma)$ is indeed independent of the chosen immersion $\iota$.

\smallskip

We next show property \ref{thmitem:r_functorial}.
We choose a factor-dominant embedding $\iota_Z : Z \inj V=D_1\times \cdots \times D_m$  into a product of hyperbolic curves. It suffices to show that
\[
\iota_Z r(\gamma_2\gamma_1)= \iota_Z r(\gamma_2)r(\gamma_1).
\]
Since $V$ is a product of hyperbolic curves and $p_i\iota_Zr(\gamma_2\gamma_1): X\to D_i$ is dominant for all $i$,  Lemma~\ref{prod_in_homcat} and Theorem~\ref{mochi} show that it suffices to prove that
\[
(\iota_Z r(\gamma_2\gamma_1))_\et= (\iota_Z r(\gamma_2)r(\gamma_1))_\et.
\]
We modify a given factor-dominant immersion $\iota: Y \inj W$,  to
\[
\iota_Y =(\iota,\iota_Zr(\gamma_2)) : Y  \inj W \times V.
\]
We set $\pr_2: W \times V \to V$ for the second projection. Using \eqref{eq:basic_compatibility_with_pi1} we can compute
\begin{align*}
(\iota_Z  r(\gamma_2 \gamma_1))_\et & = (\iota_Z)_\et(\gamma_2 \gamma_1)
 = ((\iota_Z)_\et \gamma_2)  \gamma_1  = (\iota_Z  r(\gamma_2))_\et  \gamma_1 \\
& = (\pr_2 \iota_Y)_\et \gamma_1  = (\pr_2)_\et (\iota_Y)_\et \gamma_1 \\
& = (\pr_2)_\et  (\iota_Y  r(\gamma_1))_\et = (\pr_2 \iota_Y  r(\gamma_1))_\et
 = (\iota_Z r(\gamma_2) r(\gamma_1))_\et
\end{align*}
and this shows \ref{thmitem:retraction}.

\smallskip

The established  retract property  \ref{thmitem:retraction} and functoriality \ref{thmitem:r_functorial} of $r(\gamma)$  show formally that $r(\gamma)$ is an isomorphism: the inverse $\gamma^{-1}$ gives rise to a map $r(\gamma^{-1})$ which is the inverse of $r(\gamma)$.

\smallskip

In order to show property~\ref{thmitem:r_and_map_to_hyperbolic_curve}, put $C_1=C$ and choose hyperbolic curves $C_2,\ldots,C_n$ together with a factor-dominant immersion $\iota: Y \hookrightarrow W=C_1\times \cdots \times C_n$ with $h=p_1 \iota$. Then, by \eqref{eq:basic_compatibility_with_pi1}
we have $\iota_\et r(\gamma)_\et=\iota_\et \gamma$ and composing with $(p_1)_\et$ yields the result.

\smallskip
To finish the proof we show the asserted uniqueness of $r$. Let $\iota : Y \inj W = C_1 \times \cdots \times C_n$
be a factor-dominant  immersion into a product of hyperbolic curves, denote the composite with the $i$-th projection by $f_i: Y \to C_i$, and let $\gamma : X_\et \to Y_\et$ be an isomorphism in $\Ho(\pross) \downarrow k_\et$. Clearly $r(\gamma)$ is uniquely determined by $\iota  r(\gamma)$, and even by $f_i  r(\gamma)$ for $i=1,\ldots, n$. This is uniquely determined by
Theorem~\ref{mochi} by the map
\[
(f_i  r(\gamma))_\et = f_{i,\et}  r(\gamma)_\et =  f_{i,\et}  \gamma
\]
where we used statement~\ref{thmitem:r_and_map_to_hyperbolic_curve}. This shows uniqueness of $r(\gamma)$.
\end{proof}

Another functoriality property of the retraction $r$ is the following.

\begin{proposition}\label{basiswechselfuerr}
Let $X$, $X'$, $Y$ and $Y'$ be smooth geometrically connected varieties over $k$ which can be embedded as locally closed subschemes into a product of hyperbolic curves. Assume we are given dominant $k$-morphisms $f: X'\to X$, $g: Y'\to Y$ and isomorphisms $\gamma': X'_\et \stackrel{\sim}{\to} Y'_\et$, $\gamma: X_\et \stackrel{\sim}{\to} Y_\et$ in $\Ho(\pross)\downarrow k_\et$ such that $\gamma  f_\et=g_\et  \gamma'$.
Then the following diagram commutes:
\[
\begin{tikzcd}
X'\arrow{r}{r(\gamma')}[swap]{\sim}
\arrow{d}[swap]{f}
& Y'\dar{g}\\
X\arrow{r}{r(\gamma)}[swap]{\sim}
& Y.
\end{tikzcd}
\]
\end{proposition}

\begin{proof}
Since $Y$ has an embedding into a product of hyperbolic curves, it suffices to show that
\[
h g  r(\gamma') = h  r(\gamma)  f
\]
for every dominant morphism $h: Y \to C$ to a hyperbolic curve. By Theorem~\ref{mochi}, it suffices to show that
\[
h_\et g_\et  r(\gamma')_\et = h_\et  r(\gamma)_\et  f_\et.
\]
This follows from Theorem~\ref{thm:main-retraction}~\ref{thmitem:r_and_map_to_hyperbolic_curve}:
\[
h_\et  g_\et  r(\gamma')_\et = (h g)_\et  r(\gamma')_\et
= (h g)_\et  \gamma'  = h_\et  g_\et  \gamma'
= h_\et  \gamma  f_\et = h_\et  r(\gamma)_\et  f_\et. \qedhere
\]
\end{proof}

\section{Class preservation}
\label{sec:controlpi1}

In this section we investigate the kernel of the  retraction
\[
r: \Isom_{\Ho(\pross) \downarrow k_\et}(X_\et,Y_\et) \to \Isom_k(X,Y)
\]
constructed in the proof  of Theorem~\ref{thm:main-retraction}.
Because the retraction is functorial, we may pass by composing $\gamma$ with $(r(\gamma)^{-1})_\et$ to the situation where $X = Y$ and $r(\gamma) = \id_X$.
Let $\ph=\pi_1(\gamma) \in \OutAut_{\Gal_k}(\pi_1^\et(X))$. We are going to show that $\ph$ is class preserving by elements of $\pi_1^\et(X_{\bar k})$.

\subsection{Preservation of open normal subgroups}

For this we first show that $\ph$ is a \defobjekt{normal} automorphism, i.e., every open normal subgroup is mapped by $\ph$ to itself. Note that the property of being normal is independent of the particular representative  of $\ph$ in $\Aut_{\Gal_k}\big(\pi_1^\et(X,\bp{x})\big)$.

\begin{proposition} \label{prop:NR-normal}
Let $k$ be a finitely generated field extension of\/ $\Q$ and let $X$ be a smooth geometrically connected variety over $k$
with a factor dominant embedding as a locally closed subscheme
\[
\iota : X \inj W = C_1\times \cdots \times C_n
\]
 of a product of hyperbolic curves over $k$. We assume that the following holds:
\begin{quote}
{\rm (NR):} \qquad none of the $C_i$ is rational.
\end{quote}

Then, for $\gamma \in \Aut_{\Ho(\pross) \downarrow k_\et}(X_\et)$ with $r(\gamma) = \id_X$, the induced map $\ph = \pi_1(\gamma) \in \OutAut_{\Gal_k}(\pi_1(X))$ is a normal automorphism.
\end{proposition}
\begin{proof}
Let $N\subset \pi_1^\et(X)$ be an open normal subgroup. Then also $\ph(N)$ is an open normal subgroup and we want to show that $\ph(N)=N$. We denote the connected \'{e}tale covering of $X$ associated with $N$ by $X_N$.
By Lemma~\ref{lem:finiteetale}, $(X_N)_\et$ is the covering pro-space of $X_\et$ associated with $N$.

\smallskip
Since $r(\gamma) = \id_X$, we deduce from Theorem~\ref{thm:main-retraction} \ref{thmitem:r_and_map_to_hyperbolic_curve} and  Lemma~\ref{prod_in_homcat} that in $\Ho(\pross)$
\begin{equation} \label{eq:compatibelwithiota}
\iota_\et  \gamma = \iota_\et.
\end{equation}

By Proposition~\ref{prop:covinhocat}, there  exists an isomorphism $\gamma_N: (X_N)_\et \to (X_{\ph(N)})_\et$ in $\Ho(\pross)$
such that
\begin{equation}\label{eq:phiNdiagram}
\begin{tikzcd}
(X_N)_\et \rar\dar{\gamma_N}[swap]{\wr} & X_\et \rar{\iota_\et}\dar{\gamma}[swap]{\wr}  & W_\et\dar[-, double equal sign distance]\\
(X_{\ph(N)})_\et \rar & X_\et \rar{\iota_\et}  & W_\et
\end{tikzcd}
\end{equation}
commutes in $\Ho(\pross)$.
In order to show $N=\varphi(N)$, it suffices to find an arrow such that
\[
\begin{tikzcd}
X_N \rar\dar[dashrightarrow] & X \rar{\iota}\dar[-, double equal sign distance]  & W\dar[-, double equal sign distance]\\
X_{\ph(N)} \rar & X \rar{\iota}  & W
\end{tikzcd}
\]
commutes: indeed, this implies $N\subset \varphi(N)$, hence $N=\varphi(N)$ since both have the same index in $\pi^\et_1(X)$.

Since none of the $C_i$ are rational, we can replace the $C_i$ by their smooth compactifications. After this replacement, $W$ is the product of smooth proper curves of positive genus.

Now we choose a regular connected scheme $S$ of finite type over $\bZ$ with function field $k$ such that the whole situation extends over $S$, i.e., we obtain the diagram
\[
\begin{tikzcd}
\cX_N \rar\dar[dashrightarrow] & \cX \rar{\iota}\dar[-, double equal sign distance] & \cW\dar[-, double equal sign distance]\\
\cX_{\ph(N)} \rar & \cX \rar{\iota}  & \cW.
\end{tikzcd}
\]
By generalized \v Cebotarev density \cite{Ser}, Thm.~7, the dotted arrow exists if and only if the set of closed points of $\cX$ which split completely in $\cX_N$ coincides with the set of closed points of $\cX$ which split completely in $\cX_{\ph(N)}$. We thus have to show that for any finite field $\bF$ and every $x\in \cX(\bF)$ there exists a point in $\cX_N(\bF)$ over $x$ if and only if there exists a point in $\cX_{\ph(N)}(\bF)$ over $x$.

This will be deduced as in Proposition~\ref{prop:tamagawa-argument}: Let $s\in S(\bF)$ be the image of $x$ in $S$.
We denote the fibre of $\cW$ over $s$ by $W_s$.
As in the proof of Proposition~\ref{prop:tamagawa-argument},  we may first assume that $k(s)=\bF$, and secondly
we can choose a connected \'{e}tale covering $h: W'_s \to W_s$ of $\iota(x) \in W_s$ such that $h(W'_s(\bF))=\{\iota(x)\}$. After replacing $S$ by a Nisnevich neighbourhood of $s\in S$, we can lift $h:W_s'\to W_s$ to an \'{e}tale covering $\cW'\to \cW$.

Denote the fibre products by
\[
\cX'= \cX \times_\cW \cW', \quad  \cX'_N= \cX_N \times_\cW \cW',\quad \cX'_{\ph(N)}= \cX_{\ph(N)} \times_\cW \cW'.
\]
Applying Lemma~\ref{basechange} to the commutative diagram \eqref{eq:phiNdiagram}, we conclude that the \'{e}tale homotopy types of the generic fibres $X'_N$ and $X'_{\ph(N)}$ of $\cX'_N$ and $\cX'_{\ph(N)}$ are isomorphic in
$\Ho(\pross)\downarrow k_\et$. As in the proof of Proposition~\ref{prop:tamagawa-argument}  we deduce that
\[
\cX'_N(\bF)\ne \emptyset \Leftrightarrow \cX'_{\ph(N)}(\bF)\ne \emptyset.
\]
By the choice of $\cW'$, any point of $\cX'_N(\bF)$ and of $\cX'_{\ph(N)}(\bF)$ maps to $x\in \cX(\bF)$.
Summing up, we deduce that the following statements are equivalent:
\begin{enumeral}
\item there is a point in $\cX_N(\bF)$ over $x$,
\item there is a point in $\cX_N'(\bF)$ over $x$,
\item $\cX'_N(\bF)\ne \emptyset$,
\item $\cX'_{\ph(N)}(\bF)\ne \emptyset$,
\item there is a point in $\cX'_{\ph(N)}(\bF)$ over $x$,
\item there is a point in $\cX_{\ph(N)}(\bF)$ over $x$. \qedhere
\end{enumeral}
\end{proof}

\subsection{Preservation of decomposition groups in finite quotients}

In order to talk about open subgroups (and not only about open subgroups up to conjugation), we now rigidify the situation.
Let $\bar k$ be an algebraic closure of $k$,  and let $\bp{x}: \Spec(\bar k) \to X$ be a geometric point of $X$.
Let $\gamma_0$ be a preimage of $\gamma$ under the surjection of Proposition~\ref{prop:pointed_schemes2}
\[
\begin{tikzcd}
\Aut_{\Ho(\pross_*) \downarrow (k_\et,\bar k_\et)}(X_\et,\bp{x}_\et)\arrow[->>]{r}&\Aut_{\Ho(\pross) \downarrow k_\et}(X_\et),
\end{tikzcd}
\]
i.e., $\gamma_0$ is determined by $\gamma$ up to the natural $\pi_1^\et(X_{\bar k},\bp{x})$-action. Let $\ph=\pi_1(\gamma_0)\in \Aut_{\Gal_k}(\pi_1^\et(X,\bp{x}))$. Then $\ph$ is determined by $\gamma$ up to an inner automorphism of $\pi_1^\et(X,\bp{x})$ given by an element of its subgroup $\pi_1^\et(X_{\bar k},\bp{x})$.

\begin{lemma}
Under the assumptions of Proposition~\ref{prop:NR-normal}, the automorphism $\ph : \pi^\et_1(X,\bp{x}) \to \pi^\et_1(X,\bp{x})$ sends every element $g \in \pi_1(X,\bp{x})$ to a conjugate element raised to a power:
\[
\ph(g) = h_g g^{m(g)} h_g^{-1}
\]
with $h_g \in \pi^\et_1(X,\bp{x})$ and   $m(g) \in \hZ^\times$.
\end{lemma}
\begin{proof}
Let again $N\subset \pi_1^\et(X,\bp{x})$ be an open normal subgroup. We have shown that $\ph(N)=N$, hence $\ph$ induces an automorphism $\bar{\ph}$ of
$G=\pi_1^\et(X,\bp{x})/N$. Again we choose a regular connected scheme $S$ of finite type over $\bZ$ with function field $k$ such that the whole situation extends over $S$, i.e., we obtain a Galois covering
\[
\cX_N \to \cX
\]
with Galois group $G$.

Let $g\in G$ be an arbitrary element. Our next goal is to show that $\bar{\ph}(g)$ is some power of some conjugate of $g$. For this consider the subgroup $H=\langle g \rangle \subset G$. By \v Cebotarev density, we can find a closed point $P \in \cX$ and a closed point $P'\in \cX_N$ above $P$ such that $g$ is the Frobenius of $P'$ in $G= \Aut(\cX_N/\cX)$.

Let $P_H$ be the image of $P'$ in the subextension $\cX_H$ associated with $H$. Then $P_H$ has the same residue field as $P$: $k(P_H)=k(P)$. The same argument as above shows that there is some point $Q_H\in \cX_{\bar{\ph}(H)}$ above $P$ with $k(Q_H)=k(P)$. Hence $\bar{\ph}(H)$ contains the decomposition group $G_{Q'}(\cX_N/\cX)$ of some (any) point $Q'\in \cX_N$ over $Q_H$.
Since both groups have the same order, they agree.

The group $G_{Q'}(\cX_N/\cX)$ is generated by another Frobenius over $P$, i.e., for some $h\in G$ by the conjugate $hgh^{-1}$ of $g$. We obtain
\[
\langle \bar{\ph}(g) \rangle = \bar{\ph}(H) = \langle hgh^{-1} \rangle,
\]
hence
\[
\bar{\ph}(g) = h g^m h^{-1}
\]
for some $m\in \hZ^\times$.

Now the assertion follows from the usual compactness argument as follows.
Let $g \in \pi_1(X,\bp{x})$ be an arbitrary element. For an open normal subgroup $N \subset \pi_1(X,\bp{x})$ we denote the image of $g$ in $G = \pi_1(X,\bp{x})/N$ by $g_N$. The automorphism of $G$ induced by $\ph$ will be denoted by $\ph_N$. The closed subset
\[
M_N = \{(h,m) \in \pi_1(X,\bp{x}) \times \hZ^\times  \mid \ph_N(g_N) = h_N (g_N)^m h_N^{-1}\}
\subset \pi_1(X,\bp{x}) \times \hZ^\times
\]
is non-empty by the
\v Cebotarev density argument above. If $N_1 \subset N_2$ are two open normal subgroups, then $M_{N_1} \subset M_{N_2}$. By compactness also the intersection of all $M_N$ is non-empty. Any
\[
(h_g,m(g)) \in \bigcap_N M_N \subset \pi_1(X,\bp{x}) \times \hZ^\times
\]
serves as $h_g$ and $m(g)$ as in the lemma.
\end{proof}

Since $\ph$ is a $\Gal_k$-automorphism, we can show that  the exponent $m(g)$ is always $1$:

\begin{proposition} \label{prop:NR-classpreserving}
Under the assumptions of Proposition~\ref{prop:NR-normal}, the automorphism $\ph : \pi_1^\et(X,\bp{x}) \to \pi_1^\et(X,\bp{x})$ is class preserving: every element is mapped to a conjugate element.
\end{proposition}
\begin{proof}
The cyclic cyclotomic extension of $k$ induces a surjection
\[
\pi_1^\et(X,\bp{x}) \to \Gal_k \to \hZ
\]
that is preserved by $\ph$. Therefore $m(g)  = 1$ holds for elements whose images in $\hZ$ generate an open subgroup. This is a dense set of elements in $\pi_1^\et(X,\bp{x})$ because it contains the preimage of $\bZ$.

This means that in every finite quotient we may choose the exponent equal to $1$. Applying the compactness argument again, we conclude the statement.
\end{proof}

\subsection{Rational hyperbolic factors}

We now want to drop the hypothesis (NR) (one of the assumptions of Proposition~\ref{prop:NR-normal}) from Proposition~\ref{prop:NR-classpreserving}.

\medskip
We introduce the following notation: let $k$ be a field, $X$ a connected variety over $k$ and $\ell$ a prime number.  We say a connected pointed \'{e}tale covering $h: (X',\bar x') \to (X,\bar x)$  is \defobjekt{$\ell$-geometric} if the action of  $\pi_1^\et(X_{\bar k},\bar x)$ on the geometric fibre $h^{-1}(\bar x)$ factors through a finite $\ell$-group.

If $H = \pi_1^\et(X',\bar x ') \subset \pi_1(X,\bar x)$ is the corresponding open subgroup, $\bar H = H \cap \pi_1^\et(X_{\bar k},\bar x)$ and $\bar N$ is the maximal normal subgroup of $\pi_1^\et(X_{\bar k},\bar x)$ contained in $\bar H$, then being $\ell$-geometric is equivalent to $\pi_1^\et(X_{\bar k},\bar x)/\bar N$ being an $\ell$-group.

\begin{lemma}\label{goodcov-ell}
Let $k$ be a field, $\ell$ a prime number $\not= \ch(k)$, and let $(C,\bp{c})$ be a geometrically pointed hyperbolic curve over $k$. Then
\[
\pi_1^\et(C,\bp{c}) = \bigcup \  H,
\]
where $H$ runs through the open subgroups of $\pi_1^\et(C,\bp{c})$ such that the associated covering $C_H \to C$ is $\ell$-geometric and $C_H$ has a smooth compactification of genus $\geq 1$.
\end{lemma}
\begin{proof} Let $\sigma \in \pi_1^\et(C,\bp{c})$ be arbitrary. We have to find an open subgroup $H\subset  \pi_1^\et(C,\bp{c})$ with $\sigma \in H$ such that $C_H$ is not rational and $C_H \to C$ is $\ell$-geometric.   For this we may assume that $k$ is perfect (replace $k$ by its perfect hull).
Furthermore, we can replace $k$ by an algebraic extension field such that $\langle \sigma \rangle$ surjects onto $\Gal_k$.
Then, for any $H$ containing $\sigma$, the curve $C_H$ is geometrically connected over $k$.

\smallskip
For any open subgroup $H\subset \pi_1^\et(C,\bp{c})$ we denote the boundary of a smooth compactification of $C_H$ by $S_H$; for simplicity, we write $S=S_{\pi_1^\et(C,\bp{c})}$.

We denote the genus (of a smooth compactification of) $C$ by $g(C)$.
If $g(C) > 0$ we are done, so assume that $C$ is rational. By the hyperbolicity  assumption, we have $n:= \# S (\bar k)\ge 3$, where $\bar k$ denotes  an algebraic closure of $k$.

We set  $N = \pi_1^\et(C_{\bar k},\bp{c})$. The group
\[
G = N/[N,N]N^\ell
\]
is the Galois group of the maximal elementary abelian $\ell$-covering of $C_{\bar k}$. Since $C_{\bar k}$ is isomorphic to the complement of $n$ closed points in $\P^1_{\bar k}$, this group is an $(n-1)$-dimensional $\F_\ell$-vector space and  the inertia groups in $G$ of the points in $S(\bar k)$ are pairwise distinct.
Now we consider
\[
H = \langle [N,N]N^\ell , \sigma \rangle,
\]
which is an open subgroup of $\pi_1^\et(C,\bp{c})$. The curve $C_H$ is geometrically connected over $k$ and $C_{H,\bar k} \to C_{\bar k}$ is the elementary abelian $\ell$-covering associated to the quotient $G \surj G/(\langle \sigma\rangle \cap N)$. In particular,  $C_H \to C$ is $\ell$-geometric.

Let $n_H= \# S_H(\bar k)$, and let $d$ be the degree of $C_{H,\bar k} \to C_{\bar k}$, so $d \geq \ell^{n-1}/\ell \geq \ell$. The structure of inertia groups in $G$ implies, because $\langle \sigma\rangle \cap N$ is cyclic, that all but at most one of the points in $S(\bar k)$ are ramified in $C_{H,\bar k} \to C_{\bar k}$.
This shows
\begin{equation} \label{eq:boundarycount}
n_H \leq \frac{d}{\ell}(n-1) + d.
\end{equation}
By the Riemann--Hurwitz formula,  the Euler--Poincar\'{e} characteristic is multiplicative in \'{e}tale coverings which are at most tamely ramified along the boundary of a smooth compactification. Hence we obtain
\begin{equation} \label{eq:multiplicativeEulerChar}
2g(C_H) + n_H - 2 = d \cdot (2g(C) + n - 2) = d \cdot (n - 2).
\end{equation}
If $g(C_H)>0$ we are done. Otherwise, \eqref{eq:multiplicativeEulerChar} shows $n_H \geq d(n-2) + 2 \geq \ell +2 \geq 4$. Replacing $C$ by $C_H$, we thus can assume that $n \geq 4$. Repeating this step, we either obtain $g(C_H)>0$, or $g(C_H) = 0$, and $n_H \geq d(n-2) + 2 \geq 2 \ell + 2 \geq 6$. We may therefore assume  $n \geq 6$, and in this case \eqref{eq:boundarycount} and \eqref{eq:multiplicativeEulerChar} imply
\[
2g(C_H) - 2 = d(n-2) - n_H \geq d(n-2) - \frac{d}{\ell}(n-1) - d = \frac{d}{\ell} \cdot \big((\ell-1)(n-3) - 2\big) \geq 1,
\]
showing that $C_H$ is not rational.
\end{proof}

\begin{proposition}
Let $k$ be a finitely generated field extension of\/ $\Q$ and let $X$ be a smooth geometrically connected variety over $k$
which can be embedded as a locally closed subscheme of a product of hyperbolic curves over $k$.

Then, for $\gamma \in \Aut_{\Ho(\pross) \downarrow k_\et}(X_\et)$ with $r(\gamma) = \id_X$, the induced map $\ph = \pi_1(\gamma) \in \OutAut_{\Gal_k}(\pi_1(X))$ is class preserving.
\end{proposition}
\begin{proof}
Let $\iota: X \hookrightarrow W= C_1 \times_k \cdots \times_k C_n$ be a factor-dominant embedding into a product of hyperbolic curves over $k$. We choose a geometric point $\bar x$ of $X$, and put $\bar w=\iota(\bar x)$. Let  $\sigma \in \pi_1(X,\bp{x})$ be arbitrary.

If none of the $C_i$ is rational, everything follows from Proposition~\ref{prop:NR-classpreserving}. Assume that one of the $C_i$, say $C_1$, is rational. Put $\bar c_1= p_1(\bar w)$ and let $N$ be the maximal normal subgroup of $\pi_1^\et(C_1,\bar c_1)$ contained in the image of $\pi_1^\et(X,\bar x)$ (which is open by assumption). We choose a prime number $\ell> (\pi_1^\et(C_1,\bar c_1):N)$.  Lemma~\ref{goodcov-ell} provides an $\ell$-geometric connected \'{e}tale covering $(C'_1,\bar c'_1)\to (C_1,\bar c_1)$ such that $\pi_1^\et(p_1\iota)(\sigma)\in \pi_1^\et(C'_1,\bar c'_1)$ and $C_1'$ has positive genus. Let $k'$ be the constant field of $C_1'$. Then
\[
X'= X\times_{C_1} C_1' = X_{k'} \times_{C_{1,k'}} C_1' = X_{k'} \times_{(C_{1,k'}\times_{k'} \cdots \times_{k'} C_{n,k'})} (C_1' \times_{k'} C_{2,k'} \times_{k'} \cdots \times_{k'} C_{n,k'})
\]
is geometrically connected over $k'$.
Proceeding recursively, we find a finite connected \'{e}tale covering $(W',\bp{w}')\to (W,\bp{w})$ such that

\begin{enumeral}
\item $W'=C'_1\times_{k'} \cdots \times_{k'} C'_n$, where $k'$ is a finite extension field of $k$, and the $C_i'$ are smooth geometrically connected curves over $k'$ with compactifications of genus $\ge 1$,
\item $X'=X\times_W W'$ is geometrically connected over $k'$,
\item $\sigma \in \pi_1^\et(X',\bar x')$, where $\bp{x}' := (\bp{x},\bp{w}')$.
\end{enumeral}
Let $\iota': X'\hookrightarrow W'$ be the immersion induced by $\iota$.  By Lemma~\ref{basechange}, there exists $\gamma' \in \Aut_{\Ho(\pross)}(X'_\et)$ such that the diagram
\[
\begin{tikzcd}[column sep=large]
X_\et  \arrow{d}[swap]{\gamma}& X'_\et \arrow{d}[swap]{\gamma'}\arrow{r}{\iota'_\et}  \lar& W'_\et \dar[-, double equal sign distance]\\
X_\et & X'_\et \arrow{r}{\iota'_\et} \lar&W'_\et
\end{tikzcd}
\]
commutes in $\Ho(\pross)$. By Theorem~\ref{modpi1abs}, we can lift $\gamma$ to $\gamma_0\in \Aut_{\Ho(\pross_*)}(X_\et,\bar x_\et)$ and $\gamma'$ to $\gamma'_0\in  \Aut_{\Ho(\pross_*)}(X'_\et,\bar x'_\et)$, and there exists an element $\tau \in \pi_1^\et(X,\bar x)$ with
\[
\pi_1^\et(\gamma_0)(\sigma)= \tau\pi_1^\et(\gamma'_0)(\sigma) \tau^{-1}.
\]
Furthermore, by Proposition~\ref{prop:NR-classpreserving}, there exists $\tau'\in \pi_1^\et(X',\bar x')$ with $\pi_1^\et(\gamma'_0)(\sigma)=\tau'\sigma{\tau'}^{-1}$. Hence $\pi_1^\et(\gamma_0)(\sigma)$ is a conjugate of $\sigma$. This completes the proof.
\end{proof}

In order to complete the proof of Theorem~\ref{kernel-thm}, it remains to show that $\ph$ is class preserving by elements of $\pi_1^\et(X_{\bar k})$. Since for every finite extension $k'$ of $k$ in $\bar k$, the arguments  above also apply to the automorphism $\ph'$ of $\pi_1^\et(X_{k'})$ induced by $\ph$, this follows from the next lemma.

\begin{lemma}
Let
\[
1 \lang \bar G \lang G \stackrel{p}{\lang} \Gamma \lang 1
\]
be an exact sequence of pro-finite groups and let $\ph\in \Aut_\Gamma(G)$. Assume that for every open subgroup $\Gamma'\subset \Gamma$ the induced automorphism $\ph'\in \Aut_{\Gamma'}(p^{-1}(\Gamma'))$ is class preserving. Then $\ph$ is class preserving by elements of $\bar G$.
\end{lemma}

\begin{proof}
Let $x\in G$ be arbitrary. We consider the closed subgroup $H=H_x \subset G$ generated by $x$ and $\bar G$. Since $\bar G$ is normal, the image of $H$ in $\Gamma$ is the cyclic group generated by the image of $x$.

For every open subgroup $U\subset G$ with  $H \subset U$ the set
\[
C_{U,x} = \{u \in U \mid \ph(x) = uxu^{-1}\}
\]
is nonempty and compact. Hence also the set
\[
C_{H,x} = \{h \in H \mid \ph(x) = hxh^{-1}\}  = \bigcap_{H \subset U} C_{U,x}
\]
is nonempty and therefore $\ph(x)=hxh^{-1}$ for some $h\in H$. The element $h\in H$ can be written as $h= \bar g x^m$ with $\bar g \in \bar G$ and $m \in \hat{\Z}$.
We obtain
\[
\ph(x) = (\bar{g} x^m) x (\bar{g} x^m)^{-1} = \bar g x \bar g^{-1}. \qedhere
\]
\end{proof}

\section{Strongly hyperbolic Artin neighbourhoods} \label{sec:strongArtin}

In this section we prove Theorem~\ref{anabkpi1}.

\begin{definition}\label{stronghyperbolicartin:def}
A \defobjekt{strongly hyperbolic Artin neighbourhood} is a smooth variety $X$ over $k$ such that there exists a sequence of morphisms
\[
X=X_n \to X_{n-1} \to \cdots \to X_1 \to X_0=\Spec(k)
\]
such that for all $i$
\begin{enumerate}
\item[\rm (i)] the morphism $X_i \to X_{i-1}$ is an elementary fibration into hyperbolic curves, and
\item[\rm (ii)] $X_i$ admits an embedding $X_i \hookrightarrow W_i$ into a product of hyperbolic curves.
\end{enumerate}
\end{definition}

Prominent examples of strongly hyperbolic Artin neighbourhoods are the moduli spaces $\cM_{0,n}$ of curves of genus $0$ with $n$ marked points for $n\ge 4$. The tower of elementary fibrations is the one by forgetting one marked point after the other, and the embedding into a product of hyperbolic curves comes from forgetting all marked points except for $4$ in all possible ways.

A recursive application of Proposition~\ref{prop:elemfibrationpreservesKpi1} shows that strongly hyperbolic Artin neighbourhoods over fields of characteristic zero are of type $K(\pi,1)$.

\begin{theorem}
\label{thm:goodartineighbourhoods}
Let $k$ be a finitely generated field extension of\/ $\bQ$ and let $X$ be a strongly hyperbolic Artin neighbourhood over $k$.

Let $\gamma \in \Aut_{\Ho(\pross) \downarrow k_\et} (X_\et)$ be an automorphism with $r(\gamma)=\id_X$. Then $\gamma=\id_{X_\et}$.
\end{theorem}
\begin{proof}
We choose a geometric point $\bar x$ of $X$ lying over the generic point. By Proposition~\ref{prop:pointed_schemes2}, there is a lift $\gamma_0\in \Aut_{\Ho(\pross_*) \downarrow (k_\et,\bar k_\et)} (X_\et,\bar x_\et)$ of $\gamma$. Let $\ph_0=\pi_1(\gamma_0)$. Since $X_\et$ is of type $K(\pi,1)$, it suffices to show that $\ph_0$ is an inner automorphism of $\pi_1^\et(X,\bar x)$ induced by an element of $\pi_1^\et(X_{\bar k},\bar x)$. By Theorem~\ref{kernel-thm}, $\ph_0$ is class-preserving by elements of $\pi_1^\et(X_{\bar k},\bar x)$.

We prove the theorem  by induction on the dimension of $X$. The case $\dim X=0$ is trivial, hence we may assume $\dim X\ge 1$.
Let $f:X \to Y$ be the final fibration step, i.e., an elementary fibration into hyperbolic curves with~$Y$ again a strongly hyperbolic Artin neighbourhood.  By induction, the theorem  holds for $Y$.
Let $\bar y=f(\bar x)$.
Since the higher homotopy groups of $Y$ vanish, the long exact homotopy sequence  \cite{Fr} Thm.~11.5, provides the exact sequence
\[
1 \to \pi_1^\et(X_{\bar y},\bar x) \to \pi_1^\et(X,\bar x) \to \pi_1^\et(Y, \bar y) \to 1.
\]
Because $\ph_0$ is class preserving, it preserves the normal subgroup $\Delta = \pi_1^\et(X_{\bar y},\bar x)$ and induces  a $\Gal_k$-automorphism
\[
\psi_0: \pi_1^\et(Y,\bar y) \to \pi_1^\et(Y,\bar y).
\]
Since $Y$ is of type $K(\pi,1)$, there is an element $\delta_0\in \Aut_{\Ho(\pross_*) \downarrow (k_\et,\bar k_\et)} (Y_\et,\bar y_\et)$ corresponding to $\psi_0$.
We denote by  $\delta\in\Aut_{\Ho(\pross) \downarrow k_\et} (Y_\et)$  the underlying morphism of $\delta_0$ and by $\ph=\pi_1^\et(\gamma)$ the outer group homomorphism lying under $\ph_0$.

We have $\delta f_\et = f_\et  \gamma$, and by Proposition~\ref{basiswechselfuerr}, we obtain
\[
f r(\delta) = r(\gamma)  f = f.
\]
Hence $r(\delta)= \id_Y$ and, by induction, $\delta=\id_{Y_\et}$.  Therefore $\psi_0=\pi_1(\delta_0)$ is an inner automorphism of $\pi_1^\et(Y,\bar y)$ given by an element of $\pi_1(Y_{\bar k},\bar y)$. After composing $\ph_0$ with a suitable inner automorphism given by an element of $\pi_1^\et(X_{\bar k},\bar x)$, we may assume that $\psi_0 = \id$.

\smallskip

Let $\eta \in Y$ be the generic point with residue field $K = \kappa(\eta)$, the function field of $Y$. The base change $X_K = X \times_Y \eta$ is a hyperbolic curve over $K$ and we obtain the following  diagram with exact rows
\[
\begin{tikzcd}[column sep=small]
1 \arrow{r} &  \Delta \arrow{r} \arrow[-, double equal sign distance]{d} &  \pi_1^\et(X_K,\bar x) \arrow{r} \arrow{d} &  \Gal_K \arrow{r} \arrow[->>]{d} & 1\phantom{.}\\
1 \arrow{r} &  \Delta \arrow{r} &  \pi_1^\et(X,\bar x) \arrow{r} &  \pi_1^\et(Y,\bar y) \arrow{r} & 1.
\end{tikzcd}
\]
In particular, the right square is a fibre square. Since we have arranged that the automorphism $\ph_0$ induces the identity on $\pi_1^\et(Y,\bar y)$, we may lift it as
\[
\ph_{\eta,0} = (\ph_0, \id) : \pi_1^\et(X_K,\bar x) \to \pi_1^\et(X_K,\bar x).
\]
Now we use anabelian geometry of hyperbolic curves \cite{mochizuki:localpro-p}, Thm.~A. Since with $k$ also $K$ is finitely generated over $\bQ$,
the outer isomorphism $\ph_\eta$ underlying $\ph_{\eta,0}$ comes from geometry: there is a $K$-isomorphism
\[
g_\eta : X_K \to X_K
\]
with $\ph_\eta=\pi_1^\et(g_\eta)$.
Since $X \to Y$ is an elementary fibration in hyperbolic curves, the Isom-scheme
\[
\underline{\Isom}(X/Y,X/Y) \to Y
\]
is finite and unramified by \cite{DM} Thm.~1.11.  Therefore the point
\[
g_\eta  \in \underline{\Isom}(X/Y,X/Y)(K)
\]
extends uniquely to a point
\[
g \in \underline{\Isom}(X/Y,X/Y)(Y),
\]
in other words a $Y$-isomorphism $g :  X \to X$. Since $\pi_1^\et(X_K,  \bar x) \to \pi_1^\et(X, \bar x)$ is surjective,  it follows that $\pi_1^\et(g) = \ph$, hence
\[
g_\et=\gamma
\]
in $\Ho(\pross)\downarrow k_\et$. This implies $g=r(g_\et)= r(\gamma)=\id_X$, and therefore $\gamma=g_\et=\id_{X_\et}$.
\end{proof}

\begin{proof}[Proof of Theorem~\ref{anabkpi1} and Corollary~\ref{anabkpi1kor}]
Let $X$ and $Y$ be strongly hyperbolic Artin neighbourhoods over a finitely generated field extension $k$ of\/ $\Q$. In order to prove Theorem~\ref{anabkpi1}, it suffices to show that the retraction $r$ of Theorem~\ref{main} is an inverse to the map $f \mapsto f_\et$
\[
\Isom_k(X,Y) \to \Isom_{\Ho(\pross) \downarrow k_\et}(X_\et,Y_\et).
\]
For that it suffices to show that $r$ is injective. Let $\alpha, \beta \in \Isom_{\Ho(\pross) \downarrow k_\et}(X_\et,Y_\et)$ with $r(\alpha) = r(\beta)$.  Then $r(\gamma) = \id_X$ for $\gamma = \beta^{-1} \circ \alpha$, which implies $\gamma = \id_{X_\et}$ by
Theorem~\ref{thm:goodartineighbourhoods}, hence $\alpha = \beta$. This proves Theorem~\ref{anabkpi1}.

\smallskip
Next we choose any geometric base points $\bp{x}$ of $X$ and $\bp{y}$ of $Y$. Since $(Y_\et,\bp{y}_\et)$ and $(k_\et,\bar k_\et)$ are of type $K(\pi,1)$, Proposition~\ref{prop:BGisKpi1} shows that the map
\[
\Mor_{\Ho(\pross_\ast) \downarrow (k_\et,\bar k_\et)}\big((X_\et,\bp{x}_\et),(Y_\et,\bp{y}_\et)\big)
\stackrel{\sim}{\longrightarrow}
\Mor_{\Gal_k}(\pi^\et_1(X,\bp{x}),\pi_1^\et(Y,\bp{y})).
\]
is bijective. This map is equivariant with respect to the action by $\pi_1^{\et}(Y_{\bar k},\bp{y}) \cong \pi_1^\tp(Y_{\bar k, \et},\bp{y}_\et)$ on the left hand side and composition by conjugation on the right hand side, see Lemma~\ref{lem:monodromy_acts_by_conjugation}.  By
Proposition~\ref{prop:pointed_schemes2}, the induced map on orbits restricted to isomorphisms is  the bijection
\[
\Isom_{\Ho(\pross) \downarrow k_\et}(X_\et,Y_\et)  \stackrel{\sim}{\longrightarrow}
 \OutIsom_{\Gal_k}\big(\pi_1^\et(X),\pi_1^\et(Y)\big).
\]
We conclude that Corollary~\ref{anabkpi1kor} is equivalent to Theorem~\ref{anabkpi1}.
\end{proof}
Finally, Corollary~\ref{basisofneighbourhoods} follows from the next lemma.
\begin{lemma} \label{lem:enoughstronglyhyperbolicArtinN}
Every  point of a smooth, geometrically connected variety over an infinite perfect field $k$ has a fundamental system of Zariski-open strongly hyperbolic Artin neighbourhoods.
\end{lemma}

\begin{proof}
Since the assumptions of Lemma~\ref{lem:enoughstronglyhyperbolicArtinN} carry over to open subschemes, it suffices to show that every point has a strongly hyperbolic Artin neighbourhood. We proceed by induction on the dimension of $X$.
Let $a \in X$ be a point, which we may assume to be closed.
We shrink $X$ to an affine open neighbourhood of $a$ so that $X$ becomes quasi-projective, say $X \inj \bP^n$ is an immersion. Let $x_0,\ldots, x_n$ be homogeneous linear coordinates on $\bP^n$.
Since $k$ is infinite, we can move $X$ via $\PGL_{n+1}(k)$  such that $a$ does not meet
the union $H$ of all the hyperplanes $x_i = 0$, and $x_i = x_0$ for all $i \not= 0$.
So $X \smallsetminus H$ can be embedded into a product of hyperbolic curves
\[
X \setminus H \inj \bP^n \setminus H = (\bP^1 \setminus \{0,1,\infty\})^n.
\]
We may replace $X$ by $X \setminus H$ to simplify notation.

We now argue along the lines of  \cite{sga4} XI 3.3, which is formulated over an algebraically closed field and applies without changes to an infinite perfect field and a \emph{rational} point $a \in X$. As in loc.\ cit., we choose a projective compactification $X \subseteq \bar{X} \subseteq \bP(V)$ such that $\bar{X}$ is normal (hence geometrically normal, since $k$ is perfect),  $V$ is a vector space of dimension $n+1$, and  the restriction $\dO(1)|_{\bar X}$ is the $r$-th multiple of an ample line bundle on $\bar{X}$ for some $r \geq 2$.

Let $d$ be the dimension of $X$. We consider for every linear subspace $W \subseteq V$ of dimension $d$ the linear projection $\bP(V) \dasharrow \bP(W)$, which is defined outside the linear subspace
\[
\Delta_W = \{x \in \bP(V) \ ; \ x(w) = 0 \text{ for all } w \in W\}
\]
of codimension $d$. Blowing-up $\Delta_W$ as $\sigma_W : P_W \to \bP(V)$, we obtain a morphism $\pi_W: P_W \to \bP(W)$.

Now we vary $W$ in the Grassmannian $\Grass_d(V)$ of $d$-dimensional subspaces of $V$:  Let
\[
\cW \subseteq V \times \Grass_d(V)
\]
be the universal subspace. We obtain the projection $\pi_\cW: P_\cW \to \bP(\cW)$, where  $\bP(\cW)$ is the projective space relative $\Grass_d(V)$. Let ${\bar \cU}_{\cW}$ be the closure of
\[
\cU_\cW:=X \times \Grass_d(V) \setminus \Delta_{\cW}
\]
in $P_{\cW}$, so that we obtain the diagram
\[
\begin{tikzcd}
\cU_\cW \arrow[hook]{r} \arrow[swap]{dr}{f_\cW} &  \bar{\cU}_\cW \arrow{d}{\bar f_\cW}   \\
 & \bP(\cW).
\end{tikzcd}
\]
As is shown in \cite{sga4} XI 3.3, there is a non-empty Zariski-open $\cS_{\cW} \subseteq \bP(\cW)$ such that the base change of this diagram to $\cS_\cW$ is an elementary fibration. Since $X$ embeds into the product of hyperbolic curves, it is moreover an elementary fibration into \emph{hyperbolic} curves. For every closed point $W\in \Grass_d(V)$, let $S_W \subset \bP(W)$ be the preimage of $\bP(W)$ under the projection $\cS_{\cW}\to \bP(\cW)$. Note that $S_W$ is non-empty and open in $\bP(W)$ for $W$ in an open and non-empty subscheme of $\Grass_d(V)$, namely the image of $\cS_W$ under the open map $\bP(\cW) \to \Grass_d(V)$. For those $W$, the base change to $S_W$ yields an elementary fibration
\[
\begin{tikzcd}
U_W \arrow[hook]{r} \arrow[swap]{dr}{f_W} &  \bar{U}_W \arrow{d}{\bar f_W}   \\
 & S_W.
\end{tikzcd}
\]
It therefore remains to show that every closed point $a$ of $X$ is contained in $U_W\subset X$ for some $W$ defined over $k$. Indeed,  by induction we find a strongly hyperbolic Artin neighbourhood $U_{n-1}$ of $f_W(a)$ in $S_W$. Replacing $U_W$ by $U_n = f_W^{-1}(U_{n-1})$, we are done.

The condition that $a \in U_W$ is an open condition on $W \in \Grass_d(V)$. Choosing a geometric point $\bar a$ above $a \in X$, we deduce from \cite{sga4} XI 3.3 that there is a $W_0 \in \Grass_d(V)(\bar k)$ defined over the algebraic closure $\bar k$ of $k$, such that $\bar a \in U_{W_0}$. Hence there is an open $H_{\bar a} \subset  \Grass_d(V) \times_k \bar k$ where $\bar a \in U_W$ for all $W \in H_{\bar a}$. Since $\Grass_d(V)$ is irreducible, the intersection of all Galois conjugates of $H_{\bar a}$ is a non-empty open $H \subset \Grass_d(V)$ defined  over $k$. Since $k$ is infinite and $\Grass_d(V)$ is a rational variety, we find a $k$-rational point $W \in H$. For this particular choice we have $W_{\bar k} \in H_{\bar k} \subset H_{\bar a}$, hence $\bar a \in U_{W_{\bar k}}$. But since $W_{\bar k}$ is defined over $k$, we finally have $U_W \subset X$ an open that contains $\bar a$. But then also $a \in U_W$, and by the above that concludes the induction step of the proof by induction on the dimension.
\end{proof}
\section{An absolute version of the main result} \label{sec:absolute}
Using the main theorem of birational anabelian geometry proven by F.~Pop \cite{pop:birational}, \cite{pop:alterations},
we can derive the following absolute version of Theorem~\ref{main}.
\begin{theorem}\label{main-absolute}
Let $k$ and $\ell$ be finitely generated extension fields of\/ $\Q$, and let $X/k$ and $Y/\ell$ be smooth geometrically connected varieties which can be embedded as locally closed subschemes into a product of hyperbolic curves over $k$ and $\ell$, respectively.

Then the natural map
\[
\Isom_\text{\rm Schemes}(X,Y) \longrightarrow \Isom_{\Ho(\pross)}(X_\et,Y_\et)
\]
is a split injection with a functorial retraction\/ $r$. If\/ $X$ and $Y$ are strongly hyperbolic Artin neighbourhoods, it is a bijection.
\end{theorem}

\begin{proof}
For the geometrically connected variety $X/k$, the field $k$ is uniquely determined as the maximal subfield of $\rH^0(X,\dO_X)$, see \cite{Ta} Lem.~4.2.
In particular, every isomorphism of schemes $f: X \to Y$ restricts to an isomorphism $f_{\rc} : \Spec(k) \to \Spec(\ell)$. The  assignment $f \mapsto f_{\rc}$ defines a functorial map
\[
\Isom_\text{\rm Schemes}(X,Y)  \to \Isom_\text{\rm Schemes}(\Spec(k), \Spec(\ell)).
\]

Let $\gamma : X_\et \to Y_\et$ be an isomorphism. We choose separable closures $\bar k/k$, $\bar \ell/\ell$, a geometric $k$-point
$\bp{x} :\Spec(\bar k)\to X$ and a geometric  $\ell$-point $\bp{y}: \Spec(\bar \ell)\to Y$. By Theorem~\ref{modpi1abs}, we conclude that $\gamma$ lifts to an isomorphism
\[
\gamma_0 : (X_\et,\bp{x}_\et) \longrightarrow (Y_\et,\bp{y}_\et)
\]
in $\Ho(\pross_*)$ unique up to monodromy action by $\pi_1^\tp(Y_\et,\bp{y}_\et) = \pi_1^\et(Y,\bp{y})$. In particular, we obtain an isomorphism
\[
\pi_1(\gamma_0) : \pi_1^\et(X,\bp{x}) \stackrel{\sim}{\lang} \pi_1^\et(Y,\bp{y}).
\]

Because $k$ is Hilbertian, $\Gal_k$ has no nontrivial finitely generated closed normal subgroups by \cite{FriedJarden:fieldarithmetic} Prop.~16.11.6. Since
\[
\pi_1^\et(X_{\bar k},\bp{x}) = \ker\big(\pi_1^\et(X,\bp{x}) \lang \Gal_k\big)
\]
is finitely generated by \cite{sga7} Exp.\ II Thm.~2.3.1, $\Gal_k$ is the quotient of $\pi_1^\et(X,\bp{x})$ by its maximal finitely generated normal subgroup. The same is true for $\Gal_\ell$ as a quotient $\pi_1^\et(Y,\bp{y}) \to G_\ell$. Hence $\pi_1(\gamma_0)$ induces an isomorphism $\ph_{\rc}: \Gal_k\to \Gal_\ell$ such that the following diagram commutes:
\[
\begin{tikzcd}
\pi_1^\et(X,\bp{x}) \arrow{r}{\pi_1(\gamma_0)}[swap]{\sim}\dar & \pi_1^\et(Y,\bp{y})\dar\\
\Gal_k\arrow{r}{\ph_{\rc}}[swap]{\sim} & \Gal_\ell.
\end{tikzcd}
\]
The assignment $\gamma \mapsto \ph_{\rc}$ induces a functorial map
\[
\Isom_{\Ho(\pross)}(X_\et,Y_\et) \lang \OutIsom(\Gal_k,\Gal_\ell)\cong \Isom_{\Ho(\pross)}(k_\et,\ell_\et),
\]
where the right hand isomorphism follows from  Proposition~\ref{prop:BGisKpi1} and Theorem~\ref{modpi1abs}, and determines an isomorphism $\gamma_{\rc} : k_\et \to \ell_\et$ in $\Ho(\pross)$ with $\ph_{\rc} = \pi_1(\gamma_c)$ as outer isomorphisms.

These two constructions are compatible and yield the commutative diagram (independent of the choices involved)
\[
\begin{tikzcd}
\Isom_\text{\rm Schemes}(X,Y)   \arrow{r}{(-)_\et} \dar & \Isom_{\Ho(\pross)}(X_\et,Y_\et)  \dar \\
\Isom_\text{\rm Schemes}(\Spec(k), \Spec(\ell)) \arrow{r}{(-)_\et}[swap]{\sim} & \Isom_{\Ho(\pross)}(k_\et,\ell_\et),
\end{tikzcd}
\]
Moreover, by the main theorem of birational anabelian geometry \cite{pop:birational},\cite{pop:alterations}, the bottom arrow is a bijection.

In order to prove the theorem, we may therefore fix an isomorphism $g: \Spec(\ell) \to \Spec(k)$ and restrict to isomorphisms $f$ and $\gamma$ which induce $f_{\rc} = g$ and $\gamma_{\rc} = g_\et$. We denote these sets of isomorphisms by $\Isom_g(X,Y)$ and $\Isom_{g_\et}(X_\et,Y_\et)$. We set
\[
Y' = Y \times_{\Spec(\ell)}^g \Spec(k).
\]
Then the statement of the theorem follows by applying Theorem~\ref{main} and Theorem~\ref{anabkpi1} to the bottom arrow of the commutative diagram
\[
\begin{tikzcd}
\Isom_g(X,Y)   \arrow{r}{(-)_\et} \dar[-, double equal sign distance] & \Isom_{g_\et}(X_\et,Y_\et)  \dar[-, double equal sign distance] \\
\Isom_k(X,Y') \arrow{r}{(-)_\et} & \Isom_{\Ho(\pross) \downarrow k_\et}(X_\et,Y'_\et).
\end{tikzcd}
\]
\end{proof}

We are now able to relax the geometric connectivity assumptions in Theorem~\ref{main}.

\begin{theorem}\label{main-general}
Let $k$ be a finitely generated extension field of\/ $\bQ$, and let $X$ and $Y$ be smooth varieties over $k$ such that each connected component can be embedded as a locally closed subscheme into a product of hyperbolic curves over the respective field of constants.

Then the natural map
\[
\Isom_k(X,Y) \longrightarrow \Isom_{\Ho(\pross) \downarrow k_\et}(X_\et,Y_\et)
\]
is a split injection with a functorial retraction\/ $r$. If the connected components of\/ $X$ and $Y$ are strongly hyperbolic Artin neighbourhoods over their respective fields of constants, it is a bijection.
\end{theorem}

\begin{proof}
Since isomorphisms in $\Ho(\pross)$ respect connected components, we can assume that $X$ and $Y$ are connected.
Let $K$ and $L$ be the fields of constants of $X$ and $Y$, respectively.

For  $\gamma \in \Isom_{\Ho(\pross) \downarrow k_\et}(X_\et,Y_\et)$, Theorem~\ref{main-absolute} yields an isomorphism $f = r(\gamma): X \to Y$. Let as above $f_{\rc} : \Spec(K) \to \Spec(L)$ be the induced isomorphism. It remains to show that $f_{\rc}$ is $k$-linear.

The proof of Theorem~\ref{main-absolute} first constructs an isomorphism $\gamma_{\rc} : K_\et \to L_\et$ in $\Ho(\pross)$ compatible with $\gamma$, and such that $\gamma_{\rc} = f_{\rc,\et}$.
We choose an algebraic closure $\bar k$ of $k$ and a geometric point $\bar x: \Spec(\bar k)\to X$. Let $\bar y=f(\bar x)$ and denote the induced geometric points of $\bar x^\ast : K \to \bar k$ and $\bar y^\ast : L \to \bar k$ by $\bar x$ and $\bar y$ as well.
We denote the given inclusions by $i_K: k \hookrightarrow K$ and $i_L: k \hookrightarrow L$.
Because any two algebraic closures of $k$ are $k$-isomorphic, we can further choose an isomorphism $\psi$ that makes the following diagram commutative:
\[
\begin{tikzcd}
k\arrow{r}{i_L}
\arrow[-, double equal sign distance]{d} &L\arrow{r}{\bar y^*}
& \bar k \arrow[dashed]{d}{\psi}
\\
k\arrow{r}{i_K}
& K\arrow{r}{\bar x^*}
& \bar k.
\end{tikzcd}
\]
Let $\delta: (\Spec(k),\bar x) \to (\Spec(k),\bar y)$ be the pointed scheme morphism induced by $\psi$. Furthermore, by
Theorem~\ref{modpi1abs} we may choose an isomorphism $\gamma_0 \in \Isom_{\Ho(\pross_\ast)}\big((X_\et,\bar x_\et), (Y_\et,\bar y_\et)\big)$ lifting $\gamma$.
Consider the diagram of pro-finite groups
\[
\begin{tikzcd}[column sep=large]
\pi_1^\et(X,\bar x)\arrow[twoheadrightarrow]{d}\arrow{r}{\pi_1(\gamma_0)}
&\pi_1^\et(Y,\bar y)\arrow[twoheadrightarrow]{d}\\
\pi_1(K_\et,\bar x_\et)
\arrow{r}{\pi_1^\et(f_{\rc})}
\arrow[hook]{d}&
\pi_1(L_\et,\bar y_\et)\arrow[hook]{d}\\
\pi_1(k_\et,\bar x_\et)
\arrow{r}{\pi_1^\et(\delta)=\psi^*}
& \pi_1(k_\et,\bar y_\et).
\end{tikzcd}
\]
Note that $f_{\rc,\et}$ considered as a pointed map $(K_\et,\bar x) \to (L_\et,\bar y)$ lifts $\gamma_{\rc}$. Therefore the top square commutes up to conjugation, and after replacing $\gamma_0$ by another lift $\gamma'_0$ of $\gamma$, it commutes.

Since $\gamma : X_\et\to Y_\et$ commutes with the projections to $k_\et$ in $\Ho(\pross)$,
the big square commutes up to conjugation by an element $g\in \pi_1(k_\et,\bar y_\et)$.
After replacing $\psi$ by $\psi g$, it commutes. Since the upper vertical maps are surjections, then also the lower square commutes.

The induced commutative diagram on Galois cohomology with coefficients in $\mu_n$ is by Kummer theory
\[
\begin{tikzcd}[column sep=small]
k^\times/n \arrow[-, double equal sign distance]{r} \arrow{d}{\id} & \rH^1(k,\mu_n) \arrow{d}{\delta^\ast} \arrow{r} & \rH^1(L,\mu_n)  \arrow{d}{f^\ast_{\rc}}\arrow[-, double equal sign distance]{r} & L^\times/n \arrow{d}{f^\ast_{\rc}} \\
k^\times/n \arrow[-, double equal sign distance]{r} & \rH^1(k,\mu_n) \arrow{r} & \rH^1(K,\mu_n) \arrow[-, double equal sign distance]{r} & K^\times/n.
\end{tikzcd}
\]
We obtain the congruences
\[
f_{\rc}^\ast(\alpha) \equiv \alpha  \bmod (K^\times)^n
\]
for every $\alpha \in k^\times$ and any natural number $n$.
Since finitely generated extension fields of $\bQ$ do not contain nontrivial divisible elements, we conclude that $f_{\rc}^\ast$ restricts to the identity on $k$ as claimed.
\end{proof}

\appendix
\section*{Appendix: Geometry in pro-spaces} \label{appendix:homotopyinpross}
\stepcounter{section}
This appendix deals with various aspects of pro-spaces, in particular, the existence of classifying spaces of pro-groups, the relation between pointed and unpointed homotopy equivalences and the theory of covering spaces.
The authors thank J.~Schmidt for helpful discussions on the subject.

\medskip

We will make frequent use of the fact (see \cite{EH}, 2.1.6) that, by re-indexing, every object in a pro-category $\pro\cC$ is isomorphic to a pro-object whose index category $I$ is a cofinite directed set (cofinite means that for any $i\in I$ there are only finitely many $j\in I$ with $j<i$).

\smallskip

We refer to \cite{Is} for the definition of the simplicial model structure on the category $\pross$ of pro-spaces.  The simplicial function complex is given by
\[
\Map(X,Y)_n= \Hom_{\pross}(X\times \Delta[n],Y).
\]
All objects $X$ of $\pross$ are cofibrant. If $Y$ is fibrant, then (cf.\ \cite{Hi} Prop.~9.5.24) $\Hom_{\Ho(\pross)}(X,Y)$ is given as the set of equivalence classes of elements of $\Hom_{\pross}(X,Y)$ modulo strict simplicial homotopy, i.e., deformations along the (constant) $1$-simplex $\Delta[1]$.

\subsection{Coverings of pro-spaces}\label{covappendix}
Recall (cf.\ \cite{GZ}) that a morphism of simplicial sets $p:Y\to X$ is called a covering if any commutative diagram
\[
\begin{tikzcd}
\Delta[0]\rar{u}\dar{i}& Y\dar{p}\\
\Delta[n]\rar{v}\arrow[dashed]{ru}{s}&X
\end{tikzcd}
\]
of solid arrows $u,v,i,p$ can be completed by a unique dotted arrow $s$. Coverings have the unique lifting property with respect to all horns $\Delta[n,k]\to \Delta[n]$, hence for all trivial cofibrations. In particular, they are fibrations in $\ssets$.
A covering $Y\to X$ with $Y$ and $X$ connected is called Galois covering with group $G(Y/X)=\Aut_X(Y)$ if the natural map from the quotient of $Y$ by the action of $\Aut_X(Y)$ to $X$ is an isomorphism.

If $(X,x)$ is a pointed, connected simplicial set, then there exists the universal covering simplicial set $(\tilde X,\tilde x)\to (X,x)$. Its geometric realization is the universal covering space of the geometric realization of $(X,x)$; see \cite{GZ}, Appendix I, \S3.

\begin{definition}
A morphism $Y\to X$ in $\pross$ is a \defobjekt{covering} if it is isomorphic to a level-wise covering. If\/ $Y$ and $X$ are connected, $Y\to X$ is called a \defobjekt{Galois covering} if it is isomorphic to a level-wise Galois covering.
 \end{definition}

Let $(X,x)$ be a pointed, connected pro-simplicial set. The inverse system of the pointed universal coverings of the different levels defines the pointed universal covering $(\tilde X,\tilde x)$ of $(X,x)$. The covering $\tilde X \to X$ is Galois and the fundamental group $\pi_1(X,x)$ is naturally isomorphic to the group $G(\tilde X/X)$.

\medskip
We denote the full subcategory of $\Ho(\pross)$ containing all connected pro-spaces by
\[
\Ho(\pross_*)_c .
\]
For a connected pointed pro-space $(X,x)$ and a sub pro-group $U\subset \pi_1(X,x)$, the pointed pro-covering of $(X,x)$ associated with $U$
\[
(X,x)_U \to (X,x)
\]
is well defined up to natural isomorphism in $\Ho(\pross_*)_c$. This follows from the following proposition which is proved in an analogous way as \cite{AM}, \S 2, (2.7), (2.8).

\begin{proposition} \label{prop:covinhocat}
Let $(X,x)\in \Ho(\pross_{*})_c$ and let $U\hookrightarrow \pi_1(X,x)$ be a monomorphism of pro-groups. Then there is an $(X,x)_U$ in $\Ho(\pross_{*})_c$ together with a morphism $h:(X,x)_U \to (X,x)$ characterized by the property that for each connected $(W,w)$  we have a bijection
\[
[(W,w),(X,x)_U]_{\pross_*} \xrightarrow{\sim} \{ f \in [(W,w),(X,x)]_{\pross_*} \ ; \ \pi_1(f) \text{ factors through } U\}
\]
sending $f' : (W,w) \to (X,x)_U$ to  $f = h  f'$.
\end{proposition}

In the unpointed case, the situation is more involved. We start with the following observations.

\begin{lemma}
\label{lem:covviaweak}
If\/ $f: Y \to X$ is a weak equivalence of connected pro-spaces, then the pull-back
\[
(X'\to X)\mapsto (X'\times_XY\to Y)
\]
induces an equivalence between the categories of connected coverings of\/ $X$ and of\/ $Y$.
\end{lemma}
\begin{proof} This follows  straightforward from the definition of weak equivalences in $\pross$ and standard covering theory in $\ssets$.
\end{proof}

By definition, coverings in $\pross$ have the unique lifting property with respect to level-maps which are level-wise trivial cofibrations.

\begin{lemma}
Coverings are fibrations in $\pross$.
\end{lemma}

\begin{proof}
Let $I$ be a cofinite directed index set and $(Y_i)_I \to (X_i)_I$ a level-wise covering. For any $t\in I$, the uniqueness in the defining lifting property of a covering shows that
\[
Y_t \to X_t \times_{\lim_{s<t} X_s} \lim_{s<t} Y_s
\]
is a covering, hence a fibration and a co-$1$-equivalence (in the sense of \cite{Is} Def.~3.1) in $\ssets$. We conclude that $Y\to X$ is a strong fibration in the sense of \cite{Is} Def.~6.5, hence a fibration by \cite{Is} Prop.~14.5.
\end{proof}

Since the model structure on $\pross$ is  proper (see \cite{Is} Prop.~17.1 and the correction \cite{Is3} Rmk.~4.14),
\cite{Hi} Lemma 13.3.2 yields the following.

\begin{lemma}\label{hopull}
Let $W'\to W$ be a fibration in $\pross$.
Assume that $f: X \to W$ and $g: Y \to W$ are maps in $\pross$ and $h: Y\to X$ is a weak equivalence with $g = f  h$. Then the natural map
\[
h \times \id : Y\times_W W' \to X\times_W W'
\] is a weak equivalence.
\end{lemma}

We will frequently use the fact that homotopy equivalences between pro-spaces over a common base can be base-changed along coverings of the base:

\begin{proposition}\label{basechangepross}
Let\/ $W$, $X$, $Y$ be pro-spaces and let $f:  X \to W$, $g:Y\to W$ be maps of pro-spaces. Assume that there exists
$\gamma \in \Isom_{\Ho(\pross)}(X,Y)$
such that  $g \gamma=f$ in $\Ho(\pross)$.
Let $p: W'\to W$ be a covering.

Then there exists  $\gamma'\in  \Isom_{\Ho(\pross)}(X\times_WW', Y\times_WW')$ such that the diagram
\[
\begin{tikzcd}[column sep=large]
X \arrow{d}[swap]{\gamma}&X\times_WW'\arrow{d}[swap]{\gamma'}\arrow{r} \lar& W'\dar[-, double equal sign distance]\\
Y& Y\times_WW'\arrow{r}\lar&W'
\end{tikzcd}
\]
commutes in $\Ho(\pross)$. The construction can be made functorial in $W'$ with respect to morphisms of coverings of\/ $W$ in $\pross$. In particular, if\/ $W'\to W$ is a Galois covering of connected pro-spaces, then, for all $i\ge 0$ and every abelian group $A$, the  induced isomorphisms
\[
\rH^i(Y\times_WW',A) \stackrel{(\gamma')^*}{\lang} \rH^i(X\times_WW',A)
\]
are $G(W'/W)$-equivariant.
\end{proposition}

\begin{proof}
By Lemmas~\ref{lem:covviaweak} and
\ref{hopull}, we can replace $W$ and then $X$ and $Y$ by fibrant approximations.
Hence we may assume that $\gamma: X\to Y$ is a weak equivalence in $\pross$ such that  $g \gamma=f$ in $\Ho(\pross)$, and that $p: W' \to W$ is a level-wise covering.

We choose a homotopy $F: X\times \Delta[1] \to W$ between $f$ and $g\gamma$ and denote the vertices of $\Delta[1]$ by $0$ and~$1$.  The outer square in the following diagram commutes because $F$ restricts to $f$ on $X \times \{0\}$.
\[
\begin{tikzcd}[column sep=large]
\big(X\times^f_W W'\big) \times \{0\}  \arrow[hook]{d}  \arrow{r}{\pr_{W'}}
&  W'  \arrow{d}{p}  \\
\big(X\times^f_W W'\big) \times \Delta[1]  \arrow[dashed]{ur}[swap]{F'}
\arrow{r}[swap]{F \pr_{X\times \Delta[1]}}  & W.
\end{tikzcd}
\]
The unique lifting property of a covering induces a unique map $F'$.  
Hence the assignment
\[
(x,w') \mapsto (x,F'(x,w',1))
\]
defines a  map
\[
\ph: X\times_W^fW' \lang X\times^{g\gamma}_W W'
\]
which commutes in $\Ho(\pross)$ with the respective projections to $W'$ and is an isomorphism in $\pross$ (the inverse homotopy to $F$ gives the inverse to $\ph$).
Another application of Lemma~\ref{hopull} shows that
\[
\gamma \times \id :  X \times^{g\gamma}_WW'  \lang Y\times^g_W W'
\]
is a weak equivalence. We obtain the required weak equivalence  as the composite $\gamma' = (\gamma \times \id) \ph$.  Indeed, $\gamma'$ is compatible with $\gamma$, and  $\pr_{W'} \gamma' = \pr_{W'}$ in $\Ho(\pross)$ holds, because $F'$ provides a homotopy.

The construction is obviously functorial in $W'$ and all other assertions follow immediately.
\end{proof}

\subsection{Pointed versus unpointed}
\label{sec:pointedvsunpointed}

In this section we consider the question under which conditions two connected pointed pro-spaces which are isomorphic in the unpointed homotopy category are also isomorphic in the pointed homotopy category. In general, this is not true: there are examples of connected pro-spaces whose fundamental group depends on the base point. However we will show that the problem disappears under some finiteness assumptions.

\subsubsection{Some homological algebra of limits}

For an abelian pro-group $G$, we have the derived inverse limit groups $\lim^i G$, $i\geq 0$. If $G$ is non-abelian, we have the limit group $\lim G=\lim^0 G$, and, following  \cite{BK}, XI, 6.5, the first derived inverse limit ${\lim}^1 G$, which is a pointed set. If $G$ is abelian, the Bousfield-Kan $\lim^1 G$ carries the structure of an abelian group in a natural way and coincides with the usual $\lim^1 G$.

For a pointed pro-space $(X,x)$ and  $0\le i \le j$ we thus can consider
\[
{\lim}^i \, \pi_j (X,x),
\]
which is a pointed set for $i=j=0$ and $i=j=1$, a group for $i=0$, $j=1$ and an abelian group in all other cases.
We could not find the reference for the following.

\begin{lemma} \label{higher lim of finite}
Let $G=(G_i)_{i\in I}$ be a pro-finite (resp.\ a pro-finite abelian) group. Then
\[
{\lim}^s \, G =*
\]
for $s=1$ (resp.\ for all $s\ge 1$).
\end{lemma}
\begin{proof} We may assume that the index category $I$ is a cofinite directed set.
We have a natural injection $G\hookrightarrow {G^*}$, where $G^*=(G_i^*)_{i\in I}$ is the pro-group defined by
\[
G^*_i = \prod_{j\le i} G_j
\]
with the projection maps as transition maps.
The cokernel ${G^*}/G$  is a pointed pro-finite set. We obtain an exact sequence of pointed sets, cf.\ \cite{BK} XI 6.5:
\[
*\to \lim G \to \lim {G^*} \to \lim {G^*}/G \to {\lim}^1 G \to {\lim}^1 {G^*}.
\]
The pro-group ${G^*}$ is strongly Mittag-Leffler in the sense of \cite{EH} (4.8.3), because the natural maps
\[
G_i^* = \prod_{j\le i} G_j \longrightarrow \prod_{j < i} G_j = \lim_{k<i} G^*_k
\]
are surjective. Therefore we have ${\lim}^1 {G^*}=*$ by  \cite{EH} (4.8.5). Since $G$ is pro-finite, the usual compactness argument shows that $\lim {G^*} \to \lim {G^*}/G$ is surjective. We conclude that ${\lim}^1 G =*$.

If $G$ is abelian, then so is ${G^*}$ and again by \cite{EH} (4.8.5), we have ${\lim}^s {G^*}=*$ for all $s\geq 1$. We obtain ${\lim}^s G={\lim}^{s-1} {G^*}/G$ for $s\geq 2$ and the result follows by induction.
\end{proof}

\subsubsection{Topological homotopy groups}

Let $(S^n,s_n)$ be the pointed constant pro-space given by the simplicial $n$-sphere. For a pointed pro-space $(X,x)=(X_i,x_i)_{i\in I}$, we put
\[
\pi_n^\tp(X,x):=\pi_n(\holim (X,x))=\Mor_{\Ho(\pross_*)}((S^n,s_n),(X,x))
\]
(see \cite{Is} Prop.~8.2 for the second equality), and call these the \defobjekt{topological homotopy groups} of $(X,x)$.
We call $(X,x)$ \defobjekt{path-connected} if $\pi_0^\tp(X,x)=*$.
The projections $(X,x)\to (X_i,x_i)$ induce a natural homomorphism
\[
\pi_n^\tp(X,x) \to \pi_n(X,x)
\]
from the constant group $\pi_n^\tp(X,x)$ to the pro-group $\pi_n(X,x)$ for all $n$.

\begin{theorem}\label{top=prof}
Let $(X,x)$ be a pointed pro-space such that $\pi_0(X,x)=*$ and $\pi_n(X,x)$ is pro-finite for all $n\ge 1$. Then the natural homomorphism
\[
\pi_n^\tp(X,x) \longrightarrow \lim \pi_n(X,x)
\]
is an isomorphism for all $n\geq 0$.
\end{theorem}
\begin{proof} We drop the base point $x$ from the notation. By \cite{BK}  XI, 5.2, $\holim X$ is the total space of some cosimplicial space $\Pi^* X$, i.e.\
\[
\holim X = \lim_n  \Tot^n \Pi^* X ,
\]
where
\[
\Tot^n \Pi^* X= \prod_{u\in I_n} X_{i_0}, \qquad  u=(X_{i_n}\stackrel{\alpha_n}{\to} \cdots \stackrel{\alpha_1}{\to} X_{i_0}),
\]
and the differentials $\Tot^n \Pi^* X \to \Tot^{n-1}\Pi^* X$ are defined in the usual manner.
Associated with the tower of fibrations $\Tot^n \Pi^* X$ we have the Bousfield-Kan spectral sequence
\[
E_2^{ij}= \pi^i \pi_j \Pi^* X \Rightarrow \pi_{j-i} \holim X \quad (0\leq i \leq j)
\]
($\pi^i$ are the cohomotopy groups, see \cite{BK} X, \S7.2 for the description of the $E_2$-terms). By \cite{BK} XI, \S7.1, we have
\[
\pi^i \pi_j \Pi^* X = {\lim}^i \pi_j X \quad (0\leq i \leq j).
\]
Lemma~\ref{higher lim of finite} implies $E_2^{ij}=0$ for $1\le i\leq j$, hence the inverse systems $(E_r^{ij})_r$ are constant and
\[
{\lim_r}^1 E_r^{ij} = 0 \quad (0\leq i \le j).
\]
By \cite{BK} IX, 5.1 (connectivity lemma), we conclude that $\holim X$ is connected and  \cite{BK} IX, \S5.4 (complete convergence lemma) implies that  $\lim \pi_n(X) \cong \pi_n \holim X$ for all $n\geq 1$.
\end{proof}

\subsubsection{Monodromy action on pointed homotopy classes of maps.}
\label{sec:pointedvsunpointed1}

Let \{0,1\} be the constant pro-simplicial set consisting of two points. We consider the
under category $\{0,1\} {\downarrow}\, \pross$ of pro-simplicial sets with two distinguished points together with its induced model structure, cf.\ \cite{DS}, Rmk.\,3.10.  Let
\[
I=(\Delta[1], 0,1)
\]
be the simplicial unit interval considered as a constant pro-simplicial set with the two distinguished points $0$ and $1$. For $(Y,y_0,y_1)\in \{0,1\} {\downarrow}\, \pross$ we define the set
\[
\pi_1^\tp(Y,y_0,y_1)= \Hom_{\Ho(\{0,1\} {\downarrow}\, \pross)}((\Delta[1], 0,1), (Y,y_0,y_1)).
\]
For $(X,x_0,x_1)$ in $\{0,1\} {\downarrow}\, \pross$, consider the object $ (X,x_0,x_1)\wedge I$ of $\{0,1\} {\downarrow}\, \pross$ obtained from the product $X\times I$ by collapsing  $\{x_0\} \times I$ and $\{x_1\}\times I$ to distinguished points. Then  $X\wedge I$
is a cylinder object for $X$ such that $i_0,i_1: X\to X\wedge I$ are cofibrations and the projection $X\wedge I \to X$ is a trivial fibration. Hence, if $Y$ is fibrant, then $\pi_1^\tp(Y,y_0,y_1)$ is the set of maps $u: I \to Y$ with $u(0)=y_0$, $u(1)=y_1$, modulo start and endpoint preserving homotopies.

For another point $y_2$ of $Y$, composition of paths induces a natural map
\[
\pi_1^\tp(Y,y_1,y_2) \times \pi_1^\tp(Y,y_0,y_1) \longrightarrow \pi_1^\tp(Y,y_0,y_2).
\]
The simplicial $1$-sphere $S^1$ is obtained by identifying the vertices of $\Delta[1]$. Therefore the special case $y_0=y_1$ is consistent with the definition of the topological fundamental group:
\[
\pi_1^\tp(Y,y,y)= \pi_1^\tp(Y,y).
\]

\medskip
Let $(X,x)$ be a pointed pro-space and $(Y,y_0, y_1) \in \{0,1\}{\downarrow}\,\pross$. Let $\epsilon: Y\to Z$ be a fibrant approximation and $z_i=\epsilon(y_i)$, $i=1,2$. For a path $u: I \to Z$ from $z_0$ to $z_1$ and a map $f:(X,x) \to (Z,z_0)$
we define a map $u(f): (X,x_0) \to (Z,z_1)$ by
\[
u(f) (x) = F(x,1)
\]
where $F : X \times I \to Z$ is an extension
\[
\begin{tikzcd}[column sep=small]
(X,x) \vee (I,0) \arrow[hook]{d} \arrow{rr}{u \vee f} && Z  \arrow{d} \\
X \times I  \arrow{rr} \arrow[dashrightarrow]{urr}{F}  && *
\end{tikzcd}
\]
which exists since $Z$ is fibrant. Formally in the same way as in the classical case of well-pointed topological spaces (cf.\ \cite{Whitehead}, Chapter III, \S1) one shows

\begin{lemma}\label{lem:pathaction}
The above constructions yields a well-defined composition
\[
\pi_1^\tp(Y,y_0,y_1) \times [(X,x), (Y,y_0)]_{\pross_*} \longrightarrow [(X,x), (Y,y_1)]_{\pross_*}.
\]
For another point  $y_2$ of $Y$, the composition
\[
\pi_1^\tp(Y,y_1,y_2) \times \pi_1^\tp(Y,y_0,y_1) \times [(X,x), (Y,y_0)]_{\pross_*} \longrightarrow [(X,x), (Y,y_2)]_{\pross_*}
\]
is associative. In the special case $y_0=y_1$, we obtain a natural
$\pi_1^\tp(Y,y)$-action on  $[(X,x), (Y,y)]_{\pross_*}$.
\end{lemma}

Again, formally in the same way as in the classical case of well-pointed topological spaces (cf.\ \cite{Whitehead} Chap.~III, \S1) one obtains the following:

\begin{theorem} \label{modpi1abs}
Let  $(X,x)$ and $(Y,y)$ be pointed pro-spaces and assume that $Y$ is path-connected.
Then the map induced by forgetting the base points induces a natural bijection of the orbit space
\[
\big([(X,x), (Y,y)]_{\pross_*}\big)_{\pi_1^\tp(Y,y)} \xrightarrow{\sim} [X, Y]_{\pross}
\]
with the set of morphisms of\/ $X$ to $Y$ in the unpointed homotopy category $\Ho(\pross)$.
\end{theorem}

\begin{corollary}\label{cor:pointed-unpointed-abs}
Let $(X,x)$ and $(Y,y)$ be pointed connected pro-spaces,
Assume that
\begin{enumeral}
\item $X$ and $Y$ are isomorphic in $\Ho(\pross)$,
\item $\pi_i(Y,y)$ is pro-finite for all $i\ge 1$.
\end{enumeral}
Then $(X,x)$ and $(Y,y)$ are isomorphic in the pointed homotopy category $\Ho(\pross_* )$.
\end{corollary}
\begin{proof}
Let $f: X \to Y$ be an isomorphism in $\Ho(\pross)$.  Theorem~\ref{top=prof} implies that $\pi_0^\tp(Y,y)$ is trivial, i.e., $Y$ is path-connected. Hence, by Theorem~\ref{modpi1abs} there exists a morphism $f_*: (X,x) \to (Y,y)$ in $\Ho(\pross_*)$ over $f$.  By the definition of the model structure on $\pross_*$, $f_*$ is a isomorphism.
\end{proof}

For  $f \in [(X,x), (Y,y)]_{\pross_*}$ and $\alpha \in \pi_1^\tp(Y,y)$,  the definition of the $\pi_1^\tp(Y,y)$-action implies that
\begin{equation} \label{eq:monodromy_action_is_natural}
\alpha(f)=\alpha(\id_Y)\circ f.
\end{equation}
By Lemma~\ref{lem:pathaction}, the map
\[
\pi_1^\tp(Y,y) \to \Aut_{\Ho(\pross_*)}\big((Y,y)\big), \qquad \alpha \mapsto \alpha(\id_Y),
\]
is a group homomorphism. Hence, for any functor $\mathcal{F}$ on $\Ho(\pross_*)$, there is an induced $\pi_1^\tp(Y,y)$-action on $\mathcal{F}((Y,y))$. In particular, $\pi_1^\tp(Y,y)$ acts on the set of group homomorphisms
$\ph : \pi_1^\tp(X,x) \to \pi_1^\tp(Y,y)$ by
$
\alpha(\ph) = \pi_1^\tp(\alpha(\id_Y)) \circ \ph$.

\begin{lemma} \label{lem:monodromy_acts_by_conjugation}
Let $(X,x)$ and $(Y,y)$ be pointed connected pro-spaces.
Then the $\pi_1^\tp(Y,y)$-action on
$\Hom\big(\pi_1^\tp(X,x), \pi_1^\tp(Y,y)\big)$
is by composition with conjugation: for $\alpha \in \pi_1^\tp(Y,y)$ and $\ph: \pi_1^\tp(X,x) \to \pi_1^\tp(Y,y)$ we have
\[
\alpha(\ph) (\gamma) = \alpha  \ph(\gamma)  \alpha^{-1}.
\]
In particular, the set of $\pi_1^\tp(Y,y)$-orbits is the set of outer homomorphisms:
\[
\Big(\Hom\big(\pi_1^\tp(X,x), \pi_1^\tp(Y,y)\big)\Big)_{\pi_1^\tp(Y,y)} =
\OutHom\big(\pi_1^\tp(X,x),\pi_1^\tp(Y,y)\big).
\]
\end{lemma}
\begin{proof} By functoriality, we may assume that $\ph = \id$. We may assume that $Y$ is fibrant and that $\gamma$ is represented by $g: I \to Y$, $g(0)=y=g(1)$ and $\alpha$ is represented by $a: I \to Y$, $a(0)=y=a(1)$. Then $a(g)$ represents $\alpha(\id)(\gamma)$ by \eqref{eq:monodromy_action_is_natural}.

The definition of $a(g)$ uses a map $G: I\times I \to Y$ whose restriction to the boundaries is as follows:
\[
\begin{tikzpicture}
\draw[->] (0,0) -- (1,0);
\draw[-] (1,0) -- (2,0);
\draw[->] (0,1) -- (1,1);
\draw[-] (1,1) -- (2,1);
\draw[->] (0,0) -- (0,0.5);
\draw[-] (0,0.5) -- (0,1);
\draw[->] (2,0) -- (2,0.5);
\draw[-] (2,0.5) -- (2,1);
\node at (1,-.3) {$g$};
\node at (1,1.3) {$a(g)$};
\node at (-.3,.5) {$a$};
\node at (2.3,.5) {$a$};
\node at (2.5,0) {.};
\end{tikzpicture}
\]
Hence the map $G$ provides a homotopy between $a(g)$ and $aga^{-1}$. \end{proof}

\subsubsection{Pointed versus unpointed in the relative case}
\label{sec:pointedvsunpointed2}
In Theorem~\ref{modpi1abs} we analysed the effect of forgetting the base point in $\Ho(\pross_\ast)$. Now we turn our attention to the same problem in $\Ho(\pross_\ast) \downarrow (B,b)$.

\begin{theorem} \label{modpi1secondcase}
Let $(B,b)$ be a pointed pro-space, and let $(X,x)$ and $(Y,y)$ be pointed pro-spaces over $(B,b)$.
Assume that $Y$ and $B$ are path-connected and that $\pi_1^\tp(Y,y) \to \pi_1^\tp(B,b)$ is surjective.

\begin{enumeral}
\item \label{modpi1secondcase:a}
The map induced by forgetting the base points yields a surjection
\[
\Mor_{\Ho(\pross_\ast) \downarrow (B,b)}((X,x), (Y,y)) \twoheadrightarrow
\Mor_{\Ho(\pross) \downarrow B}(X,Y).
\]
In particular, if\/  $X$ and $Y$ are isomorphic in $\Ho(\pross) \downarrow B$, then $(X,x)$ and $(Y,y)$ are isomorphic in $\Ho(\pross_*) \downarrow (B,b)$.
\item   \label{modpi1secondcase:b}

Let $S_X \subseteq \pi_1^\tp(B,b)$ be the stabilizer of the structure map $(X,x) \to (B,b)$ in $\Ho(\pross_\ast)$ and
\[
\Delta_{X,Y} \subset \pi_1^\tp(Y,y)
\]
 the preimage of\/ $S_X$ under $\pi_1^\tp(Y,y) \to \pi_1^\tp(B,b)$.

Then $\Delta_{X,Y}$ acts on $\Mor_{\Ho(\pross_\ast) \downarrow (B,b)}((X,x), (Y,y))$ and the map forgetting the base points induces a natural bijection
\[
\left(\Mor_{\Ho(\pross_\ast) \downarrow (B,b)}((X,x), (Y,y))\right)_{\Delta_{X,Y}} \xrightarrow{\sim}
\Mor_{\Ho(\pross) \downarrow B}(X,Y).
\]
\end{enumeral}
\end{theorem}

\begin{proof}
We show (a) and (b) simultaneously. Let $p_X$ and $p_Y$ be the given maps from $(X,x)$ and $(Y,y)$ to $(B,b)$.  For the moment, we consider $(X,x)$ as an absolute object, forgetting about $p_X$.
There is a disjoint union decomposition
\begin{equation} \label{eq:partionHomrelBpointed}
[(X,x),(Y,y)]_{\pross_\ast}
= \coprod_{p\in [(X,x),(B,b)]_{\pross_\ast}}
\Mor_{\Ho(\pross_\ast)  (B,b)}((X,x,p),(Y,y,p_Y)),
\end{equation}
where a morphism $f_0: (X,x) \to (Y,y)$ in $\Ho(\pross_\ast)$ maps to itself, considered as a morphism over $(B,b)$ in the $p = p_Y  f_0$-component on the right hand side. We have a similar decomposition in the unpointed case
\begin{equation} \label{eq:partionHomrelBunpointed}
[X,Y]_\pross = \coprod_{p \in [X,B]_\pross}
\Mor_{\Ho(\pross) }((X,p),(Y,p_Y)).
\end{equation}
By Theorem~\ref{modpi1abs}, we have a natural $\pi_1^\tp(Y,y)$-action on the left hand side of \eqref{eq:partionHomrelBpointed} whose orbits identify with the left hand side of \eqref{eq:partionHomrelBunpointed}.

We first show surjectivity. Let $f\in \Hom_{\Ho(\pross)\downarrow B}(X,Y)$ be given, i.e., $f$ lies in the $p_X$-component of \eqref{eq:partionHomrelBunpointed}. Let $f_0 \in [(X,x),(Y,y)]_{\pross_\ast}$ be a pre-image of $f$. By definition, $f_0$ lies in the $p_Y  f_0$-component of \eqref{eq:partionHomrelBpointed}. Since $p_Y  f_0$ and $p_X$ agree as morphisms in $\Ho(\pross)$, Theorem~\ref{modpi1abs} provides a $\sigma\in \pi_1^\tp(B,b)$ with
$\sigma (p_Y  f_0)=p_X$ in $\Hom_{\Ho(\pross_*)}((X,x),(B,b))$.
Let $\tau \in \pi_1^\tp (Y,y)$ be a pre-image of $\sigma$. Then $\tau(f_0)$ lies in the $p_X$-component of \eqref{eq:partionHomrelBpointed}. This shows surjectivity of
\[
\Mor_{\Ho(\pross_\ast) \downarrow (B,b)}((X,x), (Y,y)) \xrightarrow{}
\Mor_{\Ho(\pross) \downarrow B}(X,Y).
\]
To finish the proof, note that $f_0$ and $f'_0$ have the same image if and only if $f'_0=\tau(f_0)$ for some $\tau \in \pi_1^\tp(Y,x)$, which moreover must map to $S_X\subset \pi_1^\tp(B,b)$.

Note that in (a), if $f: X \to Y$ is an isomorphism in $\Ho(\pross) \downarrow B$ that lifts to a morphism $f_*: (X,x) \to (Y,y)$ in $\Ho(\pross_*) \downarrow (B,b)$, then by definition of the model structure on $\pross_*$, the morphism $f_*$ is also an isomorphism.
\end{proof}

\subsection{Eilenberg MacLane spaces in degree 1}
\label{sec:kpi1}

In this part of the appendix, our main goal is Proposition~\ref{prop:BGisKpi1}
which shows the existence of classifying spaces of pro-groups  in $\Ho(\pross_*)$.

\subsubsection{Pro-classifying spaces}
For a (discrete) group $G$ we consider the category with one object and automorphism group $G$ and we denote the nerve of this category by $BG$. The space $BG$ is connected and pointed by its unique $0$-cell. As is well-known, we have
\begin{equation} \label{eq:pi1ofBG}
\pi_n(BG)= \left\{\begin{array}{ll}
G, & \text{for } n=1,\\
0, & \text{for } n\ge 2,
\end{array} \right.
\end{equation}
hence, $BG$ is a functorial model for a $K(G,1)$-space.
If $f: G'\to G$ is a surjective group homomorphism, then the induced map $B(f): BG'\to BG$ is a fibration in $\ssets_*$, in particular, $BG$ is fibrant for every group~$G$.

By functoriality, this construction extends to the pro-category: for a pro-group $G$, we obtain the connected pointed pro-space $BG$ for which \eqref{eq:pi1ofBG} holds in $\progp$.
In contrast to the case of discrete groups however, the pro-space $BG$ is not a fibrant object in $\pross_*$ in general.
Nonetheless, Proposition~\ref{prop:BGisKpi1} below shows that $BG$ represents the functor $\Hom_\progp(\pi_1(-),G)$ on the pointed homotopy category. In particular, any connected pointed pro-space $(X,x)$ with  $\pi_1(X,x)=G$ and $\pi_n(X,x)=0$ for $n\ge 2$, is canonically isomorphic to $BG$ in $\Ho(\pross_*)$.

\subsubsection{A fibrant model}\label{sec:afribrant}
Our first goal is to construct a good fibrant replacement of $BG$. Let, for the moment $G$ be discrete.
Consider the category with one object for every element $g \in G$ and all objects uniquely isomorphic. We denote its nerve by $EG$. It comes along with a free right $G$-action by $g \in G$ permuting the objects $h \leadsto hg$. The functor that maps the isomorphism $h \to gh$ to the automorphism $g$ induces a $G$-covering map $EG \to BG$. The space $EG$ comes along with a natural pointing by the $0$-cell attached to the neutral element of $G$.
Moreover, $EG$ is contractible.
Again, all constructions are functorial so that we obtain a contractible pointed pro-space $EG$ associated to every pro-group $G$. For every $i\in I$ the projection $EG_i\to BG_i$ is a $G_i$-covering map.

In the following we assume without loss of generality that all occurring index categories are cofinite directed sets.
For a pro-group $G$ consider the pro-group $G^*=(G^*_i)_{i\in I}$ given by
\[
G^*_i = \prod_{j\le i} G_j
\]
with the obvious transition maps.
We have a natural injection $G\hookrightarrow G^*$ and consider the pointed pro-space
\[
B^*G :=  EG^*/G.
\]

\begin{lemma} \label{lem:fibrantreplacementforBG}
$B^\ast G$ is fibrant in $\pross_\ast$, and the natural map $BG \to B^\ast G$ is a weak equivalence.
\end{lemma}
\begin{proof}
For all $i$, the map $BG_i=EG_i/G_i \to B^*G_i= EG^*_i/G_i$ is a weak equivalence in $\semis_*$, hence $BG \to B^\ast G$ is a level-wise weak equivalence, in particular, a weak equivalence in $\pross_*$.
The pro-group $G^*$ has the property that
\[
\prod_{j\le  i} G_j= G^*_i \longrightarrow \prod_{j< i} G_j = \lim_{j<i} G_j^*
\]
is surjective, hence $\lambda_i: E G_i^\ast \to E(\lim_{j<i} G_j^*) = \lim_{j<i} EG_j^\ast$ is a fibration. In the commutative diagram
\[
\begin{tikzcd}[column sep=small]
 & E G_i^\ast \arrow{rr}{\lambda_i} \arrow{d} && \ds\lim_{j<i} EG_j^\ast  \arrow{d} & \\
(B^*G)_i  \arrow[-, double equal sign distance]{r} & EG^*_i/G_i \arrow{rr}{\mu_i} && \ds\lim_{j< i} EG^*_j/G_j \arrow[-, double equal sign distance]{r} & \lim_{j< i} (B^*G)_j
\end{tikzcd}
\]
the vertical maps are surjective coverings, in particular fibrations. Hence also $\mu_i$ is a fibration.

Furthermore, the spaces $(B^*G)_i$ and $\lim_{j< i} (B^*G)_j$, admitting contractible coverings, have trivial homotopy groups in degrees greater or equal to two.
Therefore $B^*G$ is strongly fibrant in the sense of \cite{Is} Def.~6.5, hence fibrant in $\pross$ by \cite{Is} Prop.~14.5, and in particular also fibrant in $\pross_\ast$.
\end{proof}

\subsubsection{Morphisms to pro-classifying spaces}\label{sec:morphisms-to-classifying-space}

For a pro-group $G$, we are going to show that the functor
\[
\Ho(\pross_*) \longrightarrow ({\sf sets}),\quad (X,x) \longmapsto \Hom_\progp(\pi_1(X,x),G)
\]
is represented by $BG$. The space $BG$ has exactly one $0$-cell on every level and the natural map
\begin{equation} \label{eq:BG1}
G=BG_1 \longrightarrow \pi_1( BG)=G
\end{equation}
is the identity of $G$. This identification is used to define the bottom map in diagram~\eqref{eq:BG} below.

\begin{lemma} \label{lem:BG}
For pro-groups $G$ and $G'$, all maps in the commutative diagram
\begin{equation} \label{eq:BG}
\begin{tikzcd}[column sep=small]
\Hom_\progp(G',G) \arrow{rr}{B(-)} && \Hom_{\pross_*}(BG',BG)  \arrow{d}  \\
\Hom_\progp(G',G) \arrow[-, double equal sign distance]{u} && \arrow[swap]{ll}{\pi_1(-)} \Hom_{\Ho(\pross_*)}(BG',BG) .
\end{tikzcd}
\end{equation}
are isomorphisms.
\end{lemma}

\begin{proof}
The diagram in the lemma commutes by its definition based on \eqref{eq:BG1}, in particular
\begin{equation} \label{eq:MapsofBGvsMapsofG}
\Mor_{\pross_*}(BG',BG) \to \Hom_\progp(G',G)
\end{equation}
is surjective.
By induction on $n$ and using the face maps $\partial_0,\partial_n : BG_n \to BG_{n-1}$, we see that an $n$-simplex in $BG$ is uniquely determined by its edges, i.e., its faces of dimension $1$ in $BG_1 = G$. Therefore a map $\varphi:BG'\to BG$ is uniquely determined by $\varphi_1: BG'_1 \to BG_1$ and hence the commutative diagram
\[
\begin{tikzcd}[column sep=small]
BG'_1=G' \arrow{d}{\wr} \arrow{rr}{\ph_1} && BG_1= G \arrow{d}{\wr} \\
\pi_1(BG') \arrow{rr}{\pi_1(\ph)} && \pi_1(BG)
\end{tikzcd}
\]
shows that  \eqref{eq:MapsofBGvsMapsofG}  is also injective.
It therefore remains to show that
\[
\Mor_{\pross_*}(BG',BG) \to \Mor_{\Ho(\pross_*)}(BG',BG)
\]
is surjective. This is  obvious if $BG$ is fibrant. Hence,
\[
\pi_1(-): \Mor_{\Ho(\pross_*)}(BG',BG) \to \Hom_{\progp}(G',G)
\]
is always split surjective and, moreover, an isomorphism if $BG$ is fibrant. It remains to show that $\pi_1(-)$ is injective for arbitrary $G$.

We again assume without loss of generality that all occurring index categories are cofinite directed sets. The map
\[
BG \longrightarrow B^* G
\]
of Lemma~\ref{lem:fibrantreplacementforBG} is a natural fibrant replacement.
The same argument as for $B^*G$ in Lemma~\ref{lem:fibrantreplacementforBG}  shows that  $BG^*$ is fibrant in $\pross_\ast$. Since $B^\ast G \to B G^\ast$ is a level-wise covering map, the induced map
\[
 \Mor_{\pross_*}(BG',B^\ast G) \longrightarrow  \Mor_{\pross_*}(BG',BG^\ast)
\]
is injective. Hence the left vertical map $\pi_1(-)$ in the commutative diagram
\[
\begin{tikzcd}[column sep=small]
\Mor_{\pross_*}(BG',B^*G)  \arrow[->>]{d} \arrow[hook]{rr} && \Mor_{\pross_*}(BG',BG^*) \arrow{d}{\wr} \\
\Mor_{\Ho(\pross_*)}(BG',B^*G)  \arrow{rr} && \Mor_{\Ho(\pross_*)}(BG',BG^*)\\
\Mor_{\Ho(\pross_*)}(BG',BG) \uar{\wr}\dar[->>]{\pi_1(-)} \arrow{rr} && \Mor_{\Ho(\pross_*)}(BG',BG^*)\uar[-, double equal sign distance] \dar{\pi_1(-)}[swap]{\wr}\\
\Hom_{\progp}(G',G)\arrow[hook]{rr}&&\Hom_{\progp}(G',G^*) .
\end{tikzcd}
\]
is injective. This completes the proof.
\end{proof}

\begin{proposition} \label{prop:BGisKpi1}
For a connected pointed pro-space $(X,x)$ and a pro-group $G$,  we have a functorial isomorphism
\[
\pi_1(-)  : \Mor_{\Ho(\pross_*)}((X,x),BG) \stackrel{\sim}{\longrightarrow} \Mor_\progp(\pi_1(X,x),G).
\]
\end{proposition}
\begin{proof}
For an unpointed simplicial set $X$ we denote by $B\Pi X$ the nerve of its fundamental groupoid $\Pi X$.  If $x$ is a point of $X$, the category with one object and automorphism group $\pi_1(X,x)$ is naturally a full subcategory of $\Pi X$. Hence we have an induced map of nerves $B\pi_1(X,x) \to B\Pi X$. Moreover, there is a natural map $X\to B\Pi X$.

The construction is functorial and thus we can speak about $B\Pi X$ for a pro-simplicial set $X$. If $x$ is a point of $X$, we obtain the natural map
\[
i_{X,x} : B\pi_1(X,x) \longrightarrow (B\Pi X, x).
\]
If $X$ is connected, the map $i_{X,x}$ is a weak equivalence in $\pross_*$. Composing $(X,x) \to (B\Pi X , x)$ with the inverse of $i_{X,x}$ defines a natural morphism
\[
p_{1,X}: (X,x) \lang B\pi_1(X,x)
\]
in $\Ho(\pross_*)$.
In the special case of $X = BG$ there is only one point. Then
\[
BG = B\pi_1(BG,\ast) \xrightarrow{\sim} (B\Pi BG,\ast)
\]
is an isomorphism, and $p_{1,BG} : BG \to BG$ is the identity.

Let $f: (X,x) \to BG$ be a morphism in $\Ho(\pross_\ast)$, and consider the induced map
\[
B\pi_1(f): B\pi_1(X,x) \lang B(\pi_1(BG)) = BG.
\]
Since the morphism $p_{1,X}$ is natural, the diagram
\[
\begin{tikzcd}
(X,x) \arrow{r}{f} \arrow{d}{p_{1,X}} & BG \arrow[-, double equal sign distance]{d}{p_{1,BG}} \\
B\pi_1(X,x) \arrow{r}{B\pi_1(f)} & BG
\end{tikzcd}
\]
commutes, hence $f =  B\pi_1(f) \circ p_{1,X}$ and precomposition with $p_{1,X}$ yields a surjection
\[
p_{1,X}^* : \Mor_{\Ho(\pross_*)}(B \pi_1(X,x),BG) \lang \Mor_{\Ho(\pross_*)}((X,x),BG).
\]
Next we consider the commutative diagram
\[
\begin{tikzcd}[column sep=large]
\Mor_{\Ho(\pross_*)}(B \pi_1(X,x),BG) \arrow{r}{\pi_1(-)}  \arrow{d}{p^*_{1,X}} & \Mor_\progp(\pi_1(X,x),G) \arrow[-, double equal sign distance]{d} \\
\Mor_{\Ho(\pross_*)}((X,x),BG) \arrow{r}{\pi_1(-)} & \Mor_\progp(\pi_1(X,x),G).
\end{tikzcd}
\]
Since the map $\pi_1(-)$ in the top row is a bijection by  Lemma~\ref{lem:BG}, the surjectivity of $p_{1,X}^*$ shows that all maps in the diagram are bijective.
\end{proof}

\begin{definition} Let $(X,x)$ be a pointed connected pro-space.
We call the morphism
\[
(X,x)\to B\pi_1(X,x)
\]
in $\Ho(\pross_*)$ associated via Proposition~\ref{prop:BGisKpi1} with the identity of $\pi_1(X,x)$ the \defobjekt{classifying morphism} of $(X,x)$.

We say that $(X,x)$ is of \defobjekt{type $K(\pi,1)$} if\/ $\pi_i(X,x)=0$ for all $i\geq 2$.
\end{definition}

\begin{corollary} A pointed connected pro-space $(X,x)$ is of type $K(\pi,1)$ if and only if its classifying morphism is an isomorphism in $\Ho(\pross_*)$.
\end{corollary}

\begin{proof}
By \cite{Is} Cor.~7.5, the classifying morphism $(X,x)\to B\pi_1(X,x)$ is an isomorphism if and only if it induces isomorphisms on $\pi_n$ for all $n\ge 1$. By definition, the induced map on $\pi_1$ is the identity of $\pi_1(X,x)$. Since the higher homotopy groups of $B\pi_1(X,x)$ vanish, the result follows.
\end{proof}


\end{document}